\documentclass[letterpaper]{article} 
\usepackage{aaai2026}  
\usepackage{times}  
\usepackage{helvet}  
\usepackage{courier}  
\usepackage[hyphens]{url}  
\usepackage{graphicx} 
\usepackage{natbib}  
\usepackage{caption} 
\usepackage{algorithm}
\usepackage{algorithmic}

\usepackage{newfloat}
\usepackage{listings}

\DeclareCaptionStyle{ruled}{labelfont=normalfont,labelsep=colon,strut=off} 
\floatstyle{ruled}
\newfloat{listing}{tb}{lst}{}
\floatname{listing}{Listing}
\title{Physics-informed approach for exploratory Hamilton--Jacobi--Bellman equations via policy iterations}
\author{
  Yeongjong Kim{\rm 1}\equalcontrib, Namkyeong Cho{\rm 2}\equalcontrib, Minseok Kim{\rm 3}, Yeoneung Kim{\rm 3}\thanks{Corresponding author, yeoneung@seoultech.ac.kr}
}
\affiliations{
    \textsuperscript{\rm 1}Center for Mathematical Machine Learning and its Applications (CM2LA), Pohang University of Science and Technology,\\
	\textsuperscript{\rm 2}Department of Finance and Big Data, Gachon University, \\
	\textsuperscript{\rm 3}Department of Applied Artificial Intelligence, Seoul National University of Science and Technology

}

\usepackage{xcolor}
\usepackage{multirow}
\usepackage{array}
\usepackage{multirow}
\usepackage{array}
\usepackage{cite}
\usepackage{caption}
\usepackage{subcaption}
\usepackage{comment}
\usepackage{amsmath} 
\usepackage{amssymb}  
\usepackage{amsthm}  
\usepackage{bm}
\usepackage{lipsum}
\usepackage{graphicx}
\newtheorem{theorem}{Theorem}
\newtheorem{assumption}{Assumption}

\newtheorem{lemma}{Lemma}

\newtheorem{proposition}{Proposition}

\newcommand{\R}{\mathbb{R}}

\newcommand{\de}{\,\mathrm{d}}

\begin{document}

\maketitle

\begin{abstract}
We propose a mesh-free policy iteration framework based on physics-informed neural networks (PINNs) for solving entropy-regularized stochastic control problems. The method iteratively alternates between soft policy evaluation and improvement using automatic differentiation and neural approximation, without relying on spatial discretization. We present a detailed $L^2$ error analysis that decomposes the total approximation error into three sources: iteration error, policy network error, and PDE residual error. The proposed algorithm is validated with a range of challenging control tasks, including high-dimensional linear-quadratic regulation in 5D and 10D, as well as nonlinear systems such as pendulum and cartpole problems. Numerical results confirm the scalability, accuracy, and robustness of our approach across both linear and nonlinear benchmarks.
\end{abstract}

--------------------
\section{Introduction}
Solving nonlinear Hamilton--Jacobi--Bellman (HJB) equations lies at the core of stochastic optimal control. In continuous-time settings, control strategies must handle uncertainty while ensuring long-term optimality. Among various formulations, entropy-regularized control which is also known as exploratory or soft control has emerged as a powerful paradigm that augments the running cost with a Kullback--Leibler divergence penalty against a reference measure. This regularization not only induces stochasticity in control policies but also promotes robust exploration and analytical tractability \cite{wang2020reinforcement,ziebart2008maximum}. While entropy-regularized control in continuous time~\cite{wang2020reinforcement} is rigorously formulated, the work does not address the challenges of function approximation or high-dimensional computation. Our framework complements this theory by incorporating neural approximation and residual-driven learning in a mesh-free manner.

The resulting HJB equation is a nonlinear second-order elliptic partial differential equation (PDE), typically defined over an unbounded domain. Solving such equations numerically remains a central challenge, especially in high dimensions. Soft policy iteration (PI)~\cite{tran2025policy,tang2022exploratory} alternates between policy improvement via a soft-max update and policy evaluation via solving a linear elliptic PDE. Their analysis rigorously establishes exponential convergence under two canonical regimes: bounded coefficients with small control in the diffusion, and unbounded dynamics under sufficient discounting. Their results extend classical Howard-type PI into the entropy-regularized setting with relaxed controls and randomized policies.

While the theory of soft PI is well understood, practical deployment in high-dimensional control remains difficult because each iteration requires solving a nonlinear PDE. Recent work has attempted to bypass grids with mesh-free physics-informed neural networks (PINNs), embedding the PDE residual into a variational loss \cite{lee2025hamilton,ramesh2023physics}. Yet rigorous {\it a priori} guarantees in the entropy-regularized setting are still missing. In the deterministic case, PINN-based policy iteration methods~\cite{meng2024physics,lee2025hamilton} establish only $L^\infty$ convergence on compact domains and do not incorporate entropy regularization or soft-max policies. Separately, a physics-informed model-based reinforcement learning framework~\cite{ramesh2023physics} constrains the learned dynamics model but does not formulate or analyze the underlying HJB. In contrast, our method targets the entropy-regularized HJB directly and derives an $L^{2}$ error analysis for a fully mesh-free PI scheme.

In this work, we propose a fully mesh-free implementation of the soft PI framework using PINNs. At each iteration, the value function is approximated by a neural network trained to minimize the residual of a linearized elliptic PDE. At the same time, the policy is updated analytically via the softmax formula. Crucially, this structure allows us to derive a quantitative $L^2$ energy estimate that explicitly tracks how approximation errors propagate across iterations. We decompose the total error into three interpretable sources, policy update error, PDE residual error, and policy iteration error, and provide bounds for each component. Unlike prior work, our formulation exploits the linear structure of each policy evaluation step to rigorously control approximation quality, while remaining scalable to high-dimensional domains through mesh-free optimization. This establishes a principled and practically viable foundation for solving entropy-regularized control problems in modern learning-based settings.

To summarize, our contributions are threefold:
\begin{itemize}
\item We propose a fully mesh-free implementation of the entropy-regularized policy iteration method via PINNs.
\item We derive an $L^2$-based energy estimate for the value function error under approximation of the policy distribution.
\item We present numerical experiments that demonstrate applicability to high-dimensional problems.
\end{itemize}

These results illustrate the potential of combining rigorous policy iteration theory with modern operator-learning techniques to enable reliable and scalable solvers for stochastic control problems.

\section{Problem setup}
\subsection{Exploratory HJB equations}
Let $U \subset \mathbb{R}^m$ denote the compact control space. We consider a filtered probability space $(\Omega, \mathcal{F}, \{\mathcal{F}_t\}_{t\ge0}, \mathbb{P})$ equipped with a $d$-dimensional Brownian motion $(W_t)_{t \ge 0}$.
For each relaxed control
$\pi=\{\pi(x,\cdot)\}_{x\in\mathbb{R}^d}$, where $\pi(x,\cdot)\in \mathbb{P}(U)$ is a probability measure on $U$ for each $x$,\footnote{We consider relaxed controls $\pi(x,\cdot)\in\mathbb{P}(U)$ that admit 
densities with respect to the Lebesgue measure on $U$, 
so that $\pi(x,\de u)=\pi(x,u) \de u$ and all integrals 
$\int_U (\cdot)\pi(x,u)\,du$ are well-defined.}
the state evolves according to the following SDE
\begin{equation}\label{eq:SDE-relaxed}
\begin{cases}
\de X_t^\pi  &=  \left[\int_U b(X^\pi_t,u) \pi(X_t, u)\,\mathrm{d}u \right]\de t +   \sigma(X^\pi_t) \de W_t,\\
X_0& = x \in \R^{d}. 
\end{cases}
\end{equation}
This corresponds to the situation where, at each state $X_t = x \in \R^{d}$, the control $u\in U$ is chosen randomly from $\pi(X_t, u)$, and the drift coefficient is averaged accordingly.

Let $\Pi:=\{\pi:\mathbb{R}^d\to\mathbb{P}(U)\}$.
Given policy $\pi \in \Pi$, define the value function as
\begin{equation}\label{eq:value_function_pi}\begin{aligned}
&V^\pi(x)
:= \mathbb{E}_x \left[ \int_0^\infty e^{-\rho t} \right.\\
&\quad \ \ \  \times  \left(\left.\int_U \left(r(X^\pi_t,u)-\lambda\ln\pi(X_t,u)\right)\pi(X_t,u) \de u\right) \de t\right].
\end{aligned}
\end{equation}
The optimal value function is defined as 
\[
V(x) := \sup_{\pi\in\Pi} V^{\pi}(x).
\]
Under suitable conditions, the value function solves the following nonlinear second-order elliptic PDE:
\begin{equation}\label{eq:HJB}
\begin{aligned}
&\rho V(x)=\sup_{\pi\in\Pi }F_{\pi}(x, \nabla_x V(x), D_{xx}^2 V(x))
\end{aligned}
\end{equation}
Here, for simplicity, we denote 
\begin{align*}
\Sigma(x) &:= \sigma(x)\sigma(x)^{\top}, \nonumber \\
f(x,u,p) &:=b(x,u)\cdot p +r(x,u),
\end{align*}
and
\[
\begin{split}
&F_\pi(x,p, X)\\
&:= \int_{U} (f(x,u,p)  -\lambda \ln(\pi(u) ) \pi(u) \de u+\tfrac{1}{2}\operatorname{tr}(\Sigma(x)X),
\end{split}
\]
for $x, p\in \R^d, u\in U$ and $X\in \R^{d\times d}$.

\paragraph{Notations} Let us begin with some notations used throughout the paper. For \( x \in \mathbb{R}^d \), we write \( |x| \) for the Euclidean norm. We denote the Hessian of a scalar function $f$ by $D_{xx}^2 f$. For $p \in [1,\infty)$, we write $f\in L^p(\Omega)$, if 
\[
\|f\|_{L^p(\Omega)} := \left( \int_\Omega |f(x)|^p \, \,\mathrm{d}x \right)^{1/p}.
\]
We write $C_b^k(\R^d)$ for the space of $k$-times continuously differentiable functions $f:\R^d \to \R$ such that all partial derivatives up to order $k$ are bounded. The norm is defined by
\[
\|f\|_{C_b^k} := \sum_{|\alpha| \le k} \|D^\alpha f\|_{L^\infty(\R^d)}.
\]
In particular, $f \in C_b^2(\R^d)$ implies that $f$, its gradient $\nabla_x f$, and its Hessian $D_{xx}^2 f$ are all bounded.

To ensure well-posedness of the control problem and regularity of the value function, we impose the following structural assumptions on the system dynamics and cost functions.
\begin{assumption}\label{ass:main}
We assume the following:
\begin{enumerate}
\item [\textbf{(A1)}] The matrix $\Sigma = \sigma\sigma^\top$ satisfies the uniform ellipticity condition $\Sigma(x) \succeq \frac{1}{C_0}I_d$.

\item [\textbf{(A2)}] The functions $b(\cdot,u)$, $r(\cdot,u)$, and $\sigma(\cdot)$ belong to $C_b^{2}(\R^d)$ uniformly in $u$ and 
\[
\|b(\cdot,u)\|_{C^2(\R^d)}+ \|r(\cdot,u)\|_{C^2(\R^d)}+\|\sigma(\cdot)\|_{C^2(\R^d)} <C_1
\]

\item[\textbf{(A3)}] \textbf{(Entropy regularization)}  
The constant $\lambda > 0$ is fixed and governs the strength of the entropy term in 
\eqref{eq:value_function_pi}  and \eqref{eq:HJB}.

\end{enumerate}

\end{assumption}

\subsection{Policy iteration scheme}
To solve the entropy-regularized HJB equation~\eqref{eq:HJB}, we adopt a classical policy iteration (PI) strategy extended to the randomized control setting. The key idea is to alternate between evaluating the value function under a fixed policy and improving the policy by optimizing a soft Bellman operator that reflects entropy-regularized performance.

This framework is particularly attractive because, although the original HJB equation is defined over an unbounded domain, each policy evaluation step reduces to solving a linear second-order elliptic PDE under a fixed control distribution. The policy improvement step, in turn, amounts to a pointwise softmax update that admits a closed-form expression.

In Algorithm~\ref{alg:soft-pi}, we present the exact policy iteration scheme proposed in the soft PI framework~\cite{tang2022exploratory,tran2025policy}, which we later extend to a mesh-free, PINN-based implementation. The algorithm alternates between two steps: a policy improvement step, where a soft-optimal policy is computed via the Boltzmann distribution, and a policy evaluation step, which involves solving a linear elliptic PDE under the fixed policy. The key advantage of this formulation is that the original HJB equation is decomposed into a sequence of tractable subproblems, specifically, linear PDEs with frozen coefficients and closed-form policy updates. This structure not only simplifies analysis but also lends itself naturally to function approximation methods. 

We initialize the algorithm with a value function \( v^0 \in C_b^{2}(\mathbb{R}^d) \), which may be chosen as a zero function or a smooth prior based on approximate linearization (e.g., a quadratic ansatz from LQR theory). While the initial guess affects the transient phase, the policy iteration scheme~\cite{tran2025policy,tang2022exploratory} is contractive and converges exponentially to the optimal value regardless of $v^0$.

\begin{algorithm}[H]
\caption{Soft Policy Iteration}
\label{alg:soft-pi}
\begin{algorithmic}[1]
\STATE \textbf{Initialize:} Value function $v^0 : \mathbb{R}^d \to \mathbb{R}$
\FOR{$n = 1, 2, \dots$}
    \STATE \textbf{(Policy improvement step):} Given $v^{n-1}$, define new policy $\pi^n$:

    \begin{equation}
    \label{eq:policy}
    \pi^n(x,u) 
    := 
    \frac{ \exp\left[ \tfrac{1}{\lambda} f(x,u, \nabla_x v^{n-1}(x))\right] }
    { \int_U \exp\left[ \tfrac{1}{\lambda} f(x,u', \nabla_x v^{n-1}(x))\right]   \,\mathrm{d}u' }
    \end{equation}
    \STATE \textbf{(Policy evaluation step):} Solve for $v^n$ the PDE:
    \begin{equation}
    \label{eq:evaluation}
    \rho  v^n(x)
    = F_{\pi^n}(x, \nabla_x v^{n}(x), D^2_{xx} v^n(x))
    \end{equation}
\ENDFOR
\end{algorithmic}
\end{algorithm}

For completeness, we collect the structural assumptions and key convergence results~\cite{tran2025policy,ma2024convergence}, as our $L^2$-analysis in Section~\ref{sec:error} builds on these facts.

To enable rigorous energy estimates in unbounded domains, we begin by imposing a structural decay assumption on the system coefficients. This ensures that beyond a large ball, the dynamics are effectively inactive, allowing us to restrict attention to a bounded computational domain.

\begin{assumption}\label{ass:decay}
We impose the following assumption.
\item $B := \sup_{u \in U}\|\nabla_x \cdot  b(\cdot,u)\|_{L^\infty(\mathbb{R}^d)}<\infty$, and there exists $R>0$ such that $b(x,u)=0$ uniformly in $u \in U$ for $|x|\geq R$.
\end{assumption}

The following classical $L^2$ energy estimate provides coercivity for linear elliptic equations with drift and diffusion. This lemma will be repeatedly used in our error analysis, particularly when quantifying the propagation of residuals across policy evaluation steps.

\begin{lemma}[Energy estimate for linear elliptic PDE with drift and diffusion]
\label{lem:energy_estimate}
Given $\tilde r, \tilde b, \tilde \sigma \in C_b^{2}(\R^d)$, assume that (i) $\rho > \frac{1}{2}B$, where $B := \|\nabla_x \cdot \tilde b(\cdot,u)\|_{L^\infty(\mathbb{R}^d)}$, (ii) $\tilde\Sigma(x)\succeq \frac{1}{C_0}I_d$, (iii) $\tilde r \in L^2(\mathbb{R}^d)$. Let $v$ be a unique classical solution to
\[
\rho v(x) - \tilde r(x) - \tilde b(x)\cdot \nabla_x v(x)
- \tfrac12 \mathrm{tr}(\tilde\Sigma(x) D_{xx}^2 v(x)) = 0.
\]
Then, $v \in C^2_b(\R^d)$ and we have the $L^2$ energy bound:
\[
\left(\rho - \tfrac12 B\right) \|v\|_{L^2(\R^d)}^2
+ \tfrac{1}{2C_0} \|\nabla_x v\|_{L^2(\R^d)}^2
\le \|\tilde r\|_{L^2(\R^d)} \cdot \|v\|_{L^2(\R^d)}.
\]
Therefore,
\begin{align*}
\|v\|_{L^2(\mathbb{R}^d)} &\le \frac{1}{\rho - \frac{1}{2} B} \|\tilde r\|_{L^2(\mathbb{R}^d)}, \\
\|\nabla_x v\|_{L^2(\mathbb{R}^d)} &\le \sqrt{\frac{C_0}{\rho - \frac{1}{2} B} }\|\tilde r\|_{L^2(\mathbb{R}^d)}.
\end{align*}
\end{lemma}

We next establish uniform bounds on the soft-optimal policy. This property ensures that the entropy term remains well-defined and Lipschitz continuous, which will be critical in both theoretical analysis and numerical stability.

\begin{proposition}\label{prop:bound}
Let $\{(v^n,\pi^n)\}_{n\geq0}$ be generated via Algorithm~\ref{alg:soft-pi}. Then the policy map $\pi^n$ is uniformly bounded above and below, that is, there exist $M>m>0$ such that $\pi^n \in [m,M]$.
\end{proposition}
The proof is given in Appendix~\ref{app:proof:prop}.

To control how approximation errors in the value function affect the policy, we analyze the Lipschitz continuity of the softmax map. The following result shows that the mapping from $\nabla_x v$ to the induced policy $\pi[v]$ is Lipschitz continuous in the $L^2(U)$ norm, with an explicit constant.
\begin{lemma}\label{lem:softmaxLip}
We fix a state $x\in\R^{d}$ and the control
variable $u\in U\subset\R^{m}$. Under Assumption~\ref{ass:main}, we have that
\[
\Phi(p)\\
:=
\frac{\exp \bigl[\tfrac1\lambda\bigl(b(x,u) \cdot p
                                   +r(x,u)\bigr)\bigr]}
     {\int_{U}
       \exp \bigl[\tfrac1\lambda\bigl(b(x,u') \cdot p
                                   +r(x,u')\bigr)\bigr] \,\mathrm{d}u'}
\]
satisfies 
\begin{equation}\label{eq:Lip-est}
\|\Phi(p)-\Phi(q)\|_{L^{2}(U)}
\le
L_{\Phi}   |p-q|,
\end{equation}
for all $|p|,|q|\leq \tilde L$. Here, the constant
\begin{equation}\label{eq:Lip-const}
L_{\Phi}
:=L_{\phi}:= \frac{2\|b\|_{L^\infty}
  \sup_{|p|\le\tilde L}\Phi(p)
  \sqrt{d}|U|^{1/2}}{\lambda}.
\end{equation}
\end{lemma}

\begin{proof}[Sketch of Proof]
Differentiating the softmax map \(\Phi(p)\) with respect to \(p\) reveals a Jacobian structure involving \(\Phi(p)\) itself and the control vector \(b(x,u)\). The gradient is uniformly bounded due to the boundedness of \(b\), and integration over \(U\) yields the global Lipschitz constant \(L_\Phi\). A full derivation appears in Appendix~\ref{app:lem:softmaxlip}.

\end{proof}

Having established the regularity properties of the soft-optimal policy, we now recall the global convergence guarantee for the exact soft policy iteration scheme~\cite{tran2025policy,tang2022exploratory}. This result confirms that the sequence \(\{v^n\}\) generated by Algorithm~\ref{alg:soft-pi} converges exponentially fast to the unique solution \(V\) of the entropy-regularized HJB equation. It forms the backbone of our error decomposition in Section~\ref{sec:error}, where we extend this result to the approximate PINN-based implementation.

\begin{theorem}[Convergence of policy iteration]
\label{thm:pi_global}
Suppose Assumption~\ref{ass:main} holds. Given any compact subset $\mathcal{X} \subset \R^d$, the exact policy iteration sequence $\{v^n\}$ defined by~\eqref{eq:policy} and~\eqref{eq:evaluation} converges exponentially in $L^2(\mathcal{X})$ to the unique solution $V$ of the HJB equation~\eqref{eq:HJB}. Specifically, there exists $\kappa \in (0,1)$ and $C>0$ such that
\[
\|v^n - V\|_{L^2(\mathcal{X})}
\le C_{\mathcal{X}} \kappa^n.
\]
\end{theorem}
This result confirms that under standard smoothness and ellipticity assumptions, the exact soft policy iteration method converges exponentially in $L^2$ norm to the unique solution of the entropy-regularized HJB equation. However, each iteration requires solving a linear second-order elliptic PDE, which becomes computationally prohibitive in high-dimensional settings.

To address this challenge, we now turn to a mesh-free implementation of policy iteration based on physics-informed neural networks (PINNs). This approach approximates the value function using a residual-minimizing neural network and updates the policy analytically using softmax. Crucially, it avoids spatial discretization entirely and scales favorably with dimension.

\section{Physics-informed mesh-free approach}

We now describe our proposed mesh-free implementation of the policy iteration scheme using physics-informed neural networks (PINNs). At each iteration step, the value function is approximated by a neural network trained to satisfy the linear PDE, while the policy is updated analytically via soft-max optimization. No spatial grid or basis discretization is used; all computations rely solely on randomly sampled collocation points.

\begin{algorithm}[H]
\caption{Physics-Informed Neural Network Soft Policy Iteration (PINN-SPI)}\label{alg:algorithm}
\begin{algorithmic}[1]
\STATE \textbf{Policy evaluation (value network update):}
\STATE \textbf{Input:} Value function $v^0 : \mathbb{R}^d \to \mathbb{R}$, neural network $\pi(\cdot,\cdot;\omega_0)$, $v(\cdot;\theta_0)$, tolerance $\varepsilon$.
\STATE Define residual at each sample \(x_i\) as:
\begin{align*}
&\mathcal{R}(x_i;\theta_n,\omega_n)\\ 
 &:=\rho v(x_i;\theta_n) 
- \textstyle \frac{1}{2} \operatorname{tr}(\sigma \sigma^\top D_{xx}^2v(x_i;\theta_n)) \\
& - \int_U \big[b(x_i,u)\cdot \nabla_x v(x_i;\theta_n) + r(x_i,u) \\
& \qquad - \lambda \log \pi(x_i,u;\omega_n)\big] \pi(x_i,u;\omega_n)   \,\mathrm{d}u
\end{align*}

\STATE Update \(\theta_n \to \theta_{n+1}\) by minimizing:
\[
\mathcal{L}_{\text{value}}(\theta) = \frac{1}{N} \sum_{i=1}^N |\mathcal{R}(x_i;\theta,\omega_n)|^2
\]

\STATE \textbf{Policy improvement (policy network update):}
\STATE Train \(\pi(x,u;\omega_n)\) to minimize \(\mathcal{L}_{\text{policy}}(\omega;v(\cdot;\theta_n))\)
\STATE Update \(\omega_n \to \omega_{n+1}\)

\STATE \textbf{Check convergence:} \\If
$
\frac{1}{N} \sum_{i=1}^N |v(x_i; \theta_{n+1}) - v(x_i; \theta_n)|^2 < \varepsilon$
\quad \textbf{then stop}.

\end{algorithmic}
\end{algorithm}

Let $\mathcal{X} \subset \R^d$ denote the computational domain, and let $\{x_i\}_{i=1}^N \subset \mathcal{X}$ be a set of randomly sampled training points. We fix a neural network architecture $v(\cdot;\theta):\R^d \to \R$, $\pi(\cdot,\cdot;\omega):\R^d\times \R^m \to \R$ with trainable parameters $\theta$ and $\omega$ to represent the function approximation at each step. Given $v$ and a policy $\hat\pi$ induced by $v$, we define a loss function as 
\begin{align*}
\mathcal{L}_{\text{policy}}(\hat \pi;\omega)
:=
\frac{1}{NM} \sum_{i=1}^N \sum_{j=1}^M 
\pi(x_i,u_j;\omega) \log \frac{\pi(x_i,u_j;\omega)}{\hat\pi(x_i,u_j)},
\end{align*}
to minimize the KL divergence between $\hat \pi$ and $\pi$, where
\[
\hat\pi(x,u_j) :=
\frac{ \exp\left[ \tfrac{1}{\lambda} f(x,u_j, \nabla_x v(x;\theta)) \right] }
{ \sum_{j'=1}^M \exp\left[ \tfrac{1}{\lambda} f(x,u_{j'}, \nabla_x v(x;\theta)) \right] }.
\]

The overall procedure alternates between value network training and policy updates as shown in Algorithm~\ref{alg:algorithm}. 
In Step 1--3, the value function is represented by a neural network \(v_\theta\) and optimized to minimize the residual of the linear PDE corresponding to a fixed policy \(\pi^n\).
The residual includes contributions from the drift, cost, and entropy terms and is computed pointwise over collocation samples.

In Steps 4--6, the policy is updated by fitting a neural network \(\pi(x,u;\omega)\) to approximate the softmax distribution defined by the current value network. This corresponds to minimizing a residual loss against the analytic policy improvement formula in Equation~\eqref{eq:policy}.

The process iterates until convergence in the \(L^2\) difference of the value network across iterations. The fully mesh-free nature of the method enables high-dimensional deployment without reliance on discretization or gridding.

This approach decouples the high-dimensional HJB solution into alternating supervised learning and analytical updates. The use of PINNs enables mesh-free approximation and naturally accommodates complex geometries and unstructured data. Moreover, as shown in the next subsection, the projection error from using an approximate policy $\tilde\pi^n$ in place of the exact softmax policy can be quantified using an energy estimate.

\subsection{Error estimates}
\label{sec:error}
For computation, we may restrict attention to a bounded domain $\mathcal{X} \subset \mathbb{R}^d$ that contains the effective support of the dynamics and cost functions. Specifically, under Assumption~\ref{ass:decay}, there exists $R > 0$ such that
\[
b(x,u) = 0 \quad \quad \text{for all } |x| \ge R, \quad u \in U.
\]
This allows us to define the computational domain as the ball $\mathcal{X} := B_R(0) \subset \mathbb{R}^d$ without loss of generality.

To bound the global error \(\|\tilde v^{n}-V\|_{L^2(\mathcal{X})}\) we introduce three intermediate objects and decompose the difference accordingly:
\begin{itemize}
  \item \textbf{Exact PI value} \(v^{n}\):
        the $n$th value function obtained from \emph{exact}
        policy iteration (Algorithm~\ref{alg:algorithm}) and its exact soft‑max policy \(\pi^{n}\).

  \item \textbf{Policy–consistent value} \(\hat v^{n}\):
        the exact solution of the linear PDE when the \emph{approximate}
        policy \(\tilde\pi^{n}\) is frozen, i.e.
        \(\rho\hat v^{n}=F_{\tilde\pi^{n}}(x,D\hat v^{n},D^{2}\hat v^{n})\).

  \item \textbf{PINN value} \(\tilde v^{n}\):
        the neural approximation obtained by solving the same PDE with
        finite data and finite capacity.
\end{itemize}

\paragraph{Three‑term decomposition.}
With these objects at hand we write
\[
\tilde v^{n}-V
=
\underbrace{\tilde v^{n}-\hat v^{n}}_{\displaystyle\text{(PDE error)}}
+
\underbrace{\hat v^{n}-v^{n}}_{\displaystyle\text{(policy error)}}
+
\underbrace{v^{n}-V.}_{\displaystyle\text{(iteration error)}}
\]

To evaluate the impact of approximation quality on the overall error, we define two key quantities at each iteration: the policy error $r_n$ and the residual error $q_n$.

\begin{assumption}[Policy and residual accuracy]\label{as:learn}
For each \(n\ge0\) there exist non‑negative functions
\(r_n, q_n\in L^{2}(\R^{d})\) such that
\[
\begin{split}
&r_n(x) := \|\tilde \pi^n(x,\cdot) - \hat \pi^n(x,\cdot)\|_{L^2(U)},\\
&q_n(x) := \rho\tilde v^{n}
      -F_{\tilde\pi^{n}}(x,\nabla_x \tilde v^{n},D_{xx}^{2}\tilde v^{n}),
\end{split}
\]
where $\hat \pi$ is an exact solution satisfying the policy improvement step with $\tilde v^n$. Let $r:= \sup_n \{\|r_n\|_{L^2(\mathcal{X})}\}$ and $q:=\sup_n \{\|q_n\|_{L^2(\mathcal{X})}\}$.
\end{assumption}

The quantity \(\|r_n\|_{L^2(\mathcal{X})}\)
and \(\|q_n\|_{L^{2}(\mathcal{X})}\) measure the policy improvement loss and PDE residual loss, respectively. We note here that we will derive a convergence result in terms of $r_n$ rather than the KL-divergence between $\tilde \pi^n$ and $\hat \pi^n$ as we have that 
\[
\|\tilde \pi^n - \hat \pi^n\|_{L^2(U)} \lesssim (\mathrm{D}_{\mathrm{{KL}}}(\tilde \pi^n ||\hat \pi^n))^{1/4}
\]
from the Pinsker's inequality~\cite{csiszar2011information}.

Finally, we present our main theoretical result.
\begin{theorem}[$L^2$ error]
\label{thm:expanded}
Suppose Assumptions~\ref{ass:main}-~\ref{as:learn} hold and let $\{(\tilde v^n,\tilde \pi^n)\}_{n\geq 0}$ be learned via Algorithm~\ref{alg:algorithm}. For $\rho>0$ sufficiently large so that $\gamma:=L_\Phi \tilde C_\rho \in (0,1)$, the accumulated error satisfies
\[
\begin{split}
\|\tilde v^{n}-V\|_{L^2(\mathcal X)}
\le
C(r+q)+C_{\mathcal X} \kappa^n.
\end{split}
\]
where $C_{\mathcal X}$ and $\kappa$ are from Theorem~\ref{thm:pi_global}, and $C$ is a problem-dependent constant defined in the proof.

\end{theorem}

Before we prove this theorem, we first introduce a stability result on the policy evaluation. Given a policy $\pi$, we define $b^{\pi}(x) := \int_U b(x,u)\pi(x,u) \,\mathrm{d}u$, $f^{\pi}(x) := \int_U f(x,u)\pi(x,u) \,\mathrm{d}u$, and $\mathcal{H}^{\pi}(x) := \int_U \log\pi(x,u)\cdot\pi(x,u) \,\mathrm{d}u$. 
\begin{lemma}[Local policy stability on $B_R$]\label{lem:local-stability}
Suppose Assumption~\ref{ass:main} and \ref{ass:decay} hold. Fix $R>0$ and denote $\mathcal{X}:=B_R(0)=\{x\in\R^{d}: |x|\leq R\}$ with
boundary $\Gamma:=\partial\mathcal{X}$.  
Let $\pi,\tilde\pi\in\mathbb{P}(U)$ be two policies satisfying $\pi,\tilde \pi \in [m,M]$ for some $m,M>0$, and
let $v^{\pi},v^{\tilde\pi}$ solve
\[
\rho v^{\pi}
  =b^{\pi} \cdot\nabla v^{\pi}
   +\tfrac12\operatorname{tr}(\Sigma D^{2}v^{\pi})
   +f^{\pi}-\lambda\mathcal H^{\pi}\quad\text{in}\quad\mathcal{X},
\]
and
\[
\rho \tilde v^{\tilde \pi}
  =b^{\pi} \cdot\nabla \tilde v^{\tilde \pi}
   +\tfrac12\operatorname{tr}(\Sigma D_{xx}^{2}v^{\tilde \pi})
   +f^{\tilde \pi}-\lambda\mathcal H^{\tilde \pi}\quad\text{in}\quad\mathcal{X},
\]
respectively. Then
\[
   \|\tilde v -v\|_{L^{2}(\mathcal{X})}
   \le
   \tilde C_\rho
   \|\tilde\pi-\pi\|_{L^{2}(\mathcal{X}\times U)},
\]
and
\[
\|\nabla_x \tilde v -\nabla v \|_{L^2(\mathcal X)} \leq \tilde C_\rho \|\tilde \pi-\pi\|_{L^2(\mathcal{X}\times U)},
\]
where $
\tilde C_\rho:=\max\{\frac{\sqrt{C}}{\rho-\frac{1}{2}B},C\sqrt{\frac{C_0}{\rho-\frac{1}{2}B}} \}$
and $C$ is a problem-dependent constant defined in the proof.
\end{lemma}
The proof of the lemma is presented in Appendix~\ref{lem:lip:proof}. We continue to give a proof of Theorem~\ref{thm:expanded}.
\begin{proof}[Proof of Theorem~\ref{thm:expanded}]
We begin by decomposing the total error via triangle inequality:
\[
\begin{split}
&\|\tilde v^n - V\|_{L^2(\mathcal{X})}\\
&\le
\underbrace{\|\tilde v^n - \hat v^n\|_{L^2(\mathcal{X})}}_{\text{(I)}}
+
\underbrace{\|\hat v^n - v^n\|_{L^2(\mathcal{X})}}_{\text{(II)}}
+
\underbrace{\|v^n - V\|_{L^2(\mathcal{X})}}_{\text{(III)}},
\end{split}
\]
where $\hat v^n := T[\tilde \pi^n]$ is the exact PDE solution under the learned policy.

By Lemma~\ref{lem:energy_estimate}, we have
\[
\|\tilde v^n - \hat v^n\|_{L^2(\mathcal{X})}\leq C_\rho \|q_n\|_{L^2(\mathcal{X})}.
\]

To estimate (II), we invoke Lemma~\ref{lem:local-stability} to deduce that 
\[
\begin{split}
\|\hat v^n - v^n\|_{L^2(\mathcal{X})}
&= \|T[\tilde \pi^n] - T[\pi^n]\|_{L^2(\mathcal{X})}\\
&\le \tilde C_{\rho} \|\tilde \pi^n - \pi^n\|_{L^2(\mathcal{X} \times U)}.
\end{split}
\]

We now split the policy gap:
\begin{equation}\label{eq:policy_gap}
\|\tilde \pi^n - \pi^n\|_{L^2(U)} \leq 
\underbrace{\|\tilde \pi^n - \hat \pi^n\|_{L^2(U)}}_{=r_n(x)}
+
\underbrace{\|\hat \pi^n - \pi^n\|_{L^2(U)}}_{\text{(analytic softmax mismatch)}},
\end{equation}
where $\hat \pi^n := \pi[\tilde v^{n-1}]$, $\pi^n := \pi[v^{n-1}]$. 

The softmax map \(\pi[v] = \Phi(\nabla_x v)\) is Lipschitz in gradient as demonstrated in Proposition~\ref{lem:softmaxLip}:
    \[
    \|\tilde \pi - \pi\|_{L^2(U)} \le L_\Phi |\nabla_x \tilde v^{n-1} - \nabla_x v^{n-1}|.
    \]
Therefore, we have that
\[
\begin{split}
\|\tilde \pi^n - \pi^n\|_{L^2(\mathcal{X}\times U)}
&= \|\Phi(\nabla_x \tilde v^{n-1}) - \Phi(\nabla_x v^{n-1})\|_{L^2(\mathcal{X}\times U)}\\
&\le L_\Phi \|\nabla_x \tilde v^{n-1} - \nabla_x v^{n-1}\|_{L^2(\mathcal{X})}.
\end{split}
\]

Again split:
\[
\begin{split}
&\|\nabla_x \tilde v^{n-1} - \nabla_x v^{n-1}\|_{L^2(\mathcal{X})}\\
&\le \|\nabla_x \tilde v^{n-1} - \nabla_x \hat v^{n-1}\|_{L^2(\mathcal{X})}
+ \|\nabla_x \hat v^{n-1} - \nabla_x v^{n-1}\|_{L^2(\mathcal{X})}.
\end{split}
\]

Apply the residual bound and Lemma~\ref{lem:energy_estimate}, 
\[
\|\nabla_x \tilde v^{n-1} - \nabla_x \hat v^{n-1}\|_{L^2(\mathcal{X})} \le C_\rho \|q_{n-1}\|_{L^2(\mathcal{X})},
\]
and by Lemma~\ref{lem:local-stability},
\[
\|\nabla_x \hat v^{n-1} - \nabla_x v^{n-1}\|_{L^2(\mathcal{X})} \le \tilde C_{\rho} \|\tilde \pi^{n-1} - \pi^{n-1}\|_{L^2(\mathcal{X}\times U)}.
\]

By iteratively applying~\eqref{eq:policy_gap} with 
\[
k_n:= \|\tilde \pi^n - \pi^n\|_{L^2(\mathcal{X}\times U)},
\]
we have that
\[
k_n \leq r+L_\Phi C_\rho q + \underbrace{L_\Phi \tilde C_\rho}_{=\gamma} k_{n-1},
\]
which leads to
\[
k_n \leq  k_0 \gamma^n+\underbrace{\sum_{i=0}^{n-1} ( r+L_\Phi C_\rho q)\gamma^{i}.}_{\leq \frac{r+L_\Phi C_\rho q}{1-\gamma}}
\]

Therefore, 
\[
\begin{split}
\|\hat v^{n}-v^{n}\|_{L^2(\mathcal X)} &\le \tilde C_{\rho} (r+ L_\Phi(C_\rho q + k_{n-1}) )\\
&\leq \tilde C_\rho( r+ L_\Phi C_\rho q+ L_\Phi k_{n-1})\\
&\leq C(r+q).
\end{split}
\]

Putting altogether and recalling Theorem~\ref{thm:pi_global},
\[
\|\tilde v^{n}-V\|_{L^2(\mathcal X)}
\le
C(r+q)+C_{\mathcal X} \kappa^n.
\]
\end{proof}

An important implication of this theorem is that the total approximation error does not accumulate over policy iterations. Instead, it remains uniformly bounded in terms of the residual and policy approximation errors.

In the next section, we validate these theoretical findings through numerical experiments.

\section{Experiments}
We evaluate the proposed PINN-based soft policy iteration (PINN-SPI) framework on a suite of entropy-regularized stochastic control problems, ranging from low-dimensional nonlinear systems to high-dimensional linear-quadratic regulators (LQR). Our goals are to  (i) demonstrate scalability to high-dimensional settings, (ii) demonstrate the monotonicity property~\cite[Corollary 5.1]{tran2025policy}. All implementation details, hyperparameter configurations, and reproducibility materials are available in our public repository.\footnote{\url{https://github.com/tomatofromsky/pinn-spi-aaai2026}}

\subsection{Linear-quadratic regulator (LQR) with compact action space}
We consider entropy-regularized linear-quadratic regulator (LQR) problems in 5, 10 and 20 dimensions with compact action constraints. The system dynamics are
\[
\mathrm{d} X_t = (A X_t + B u_t)\, \mathrm{d}t + \sigma\, \mathrm{d}W_t,
\]
where the reward is given by
\[
L(x, u) = - x^\top Q x - u^\top R u,
\]
and our control set $U$ is defined as
\[
u \in U := \{ u \in \mathbb{R}^m \mid \|u\|_\infty \le \textbf{u} \}.
\]

Unlike classical LQR problems, where Riccati equations yield closed-form solutions, the compact control constraint requires direct numerical solution of the HJB equation. The value function, however, remains close to quadratic, making it a suitable benchmark for evaluating approximation quality and convergence.

We apply PINN-SPI and compare its performance with the model-free Soft Actor-Critic (SAC)~\cite{haarnoja2018soft} and Proximal Policy Optimization (PPO)~\cite{schulman2017proximal} algorithm. Both methods are initialized with the same linear controller derived from the unconstrained problem. Our approach uses residual minimization with randomly sampled collocation points and softmax-based policy updates.

\paragraph{Setup.}Experiments are conducted with randomly generated matrices $(A,B,Q,R)$, and set isotropic noise $\sigma = 0.1 I_d$ where $I_d$ denotes $d\times d$ identity matrix. We take $d = 5,10$ and 20, and set $\textbf{u}=10$. Evaluation metrics include average discounted reward per iteration.

\paragraph{Results.} Figure~\ref{fig:lqr:comp} compares PINN-SPI against SAC and PPO in 5D and 10D LQR settings. PINN-SPI consistently achieves higher reward and smoother convergence compared to other methods.

Figure~\ref{fig:lqr:mono} further illustrates the convergence behavior of PINN-SPI. The evaluation reward increases monotonically over training time, confirming the theoretical stability of policy iteration. The 10D case in particular highlights the scalability of our mesh-free implementation. An additional experiment for 20D case is provided in Appendix~\ref{app:exp}.

\begin{figure}[htbp]
  \centering
  \begin{subfigure}[b]{0.4\textwidth}
    \centering
    \includegraphics[width=\textwidth]{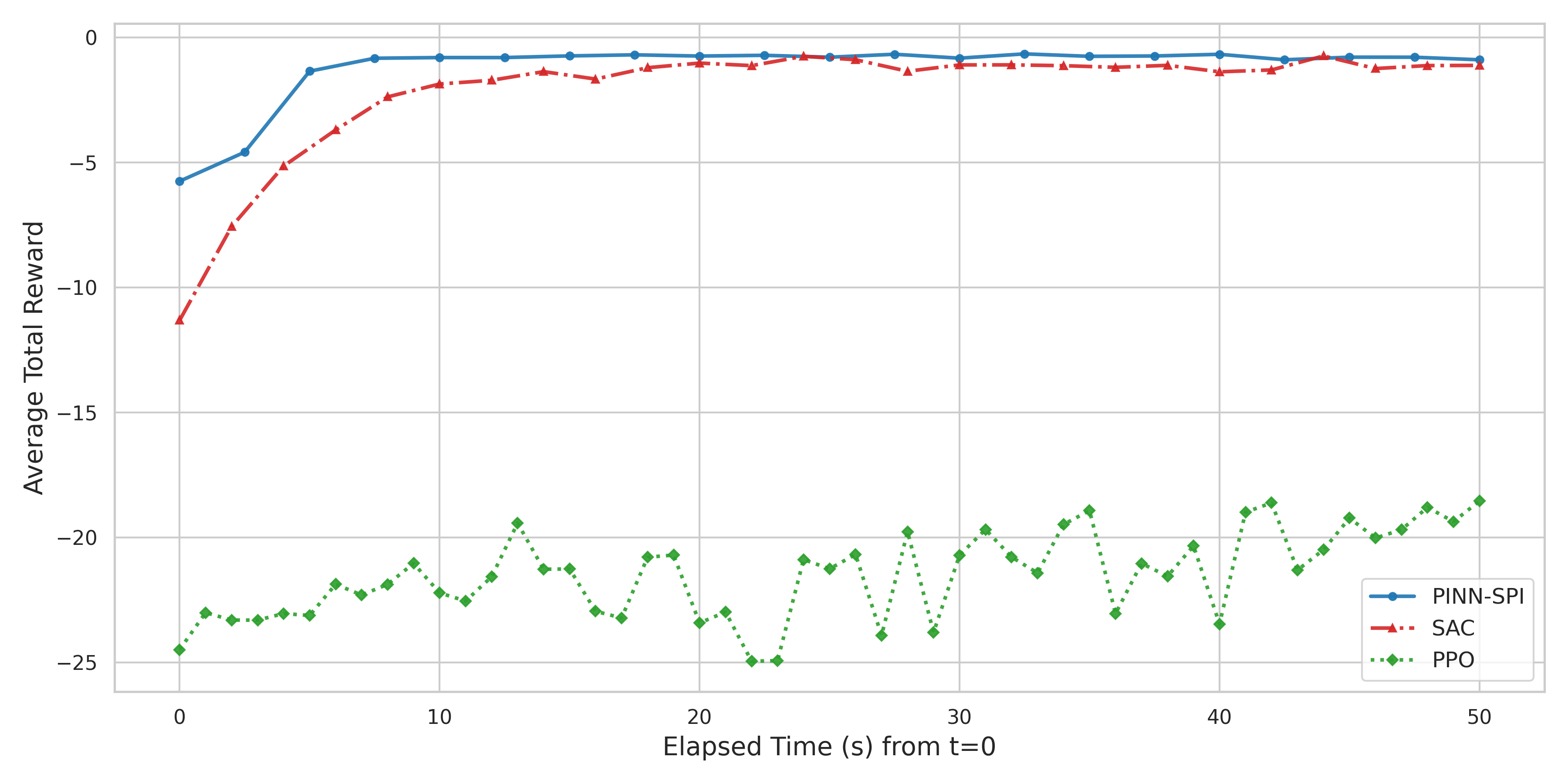}
    \caption{5D stochastic LQR with compact control set and entropy regularization.}
    \label{fig:img1-1}
  \end{subfigure}
  \hfill
  \begin{subfigure}[b]{0.4\textwidth}
    \centering
    \includegraphics[width=\textwidth]{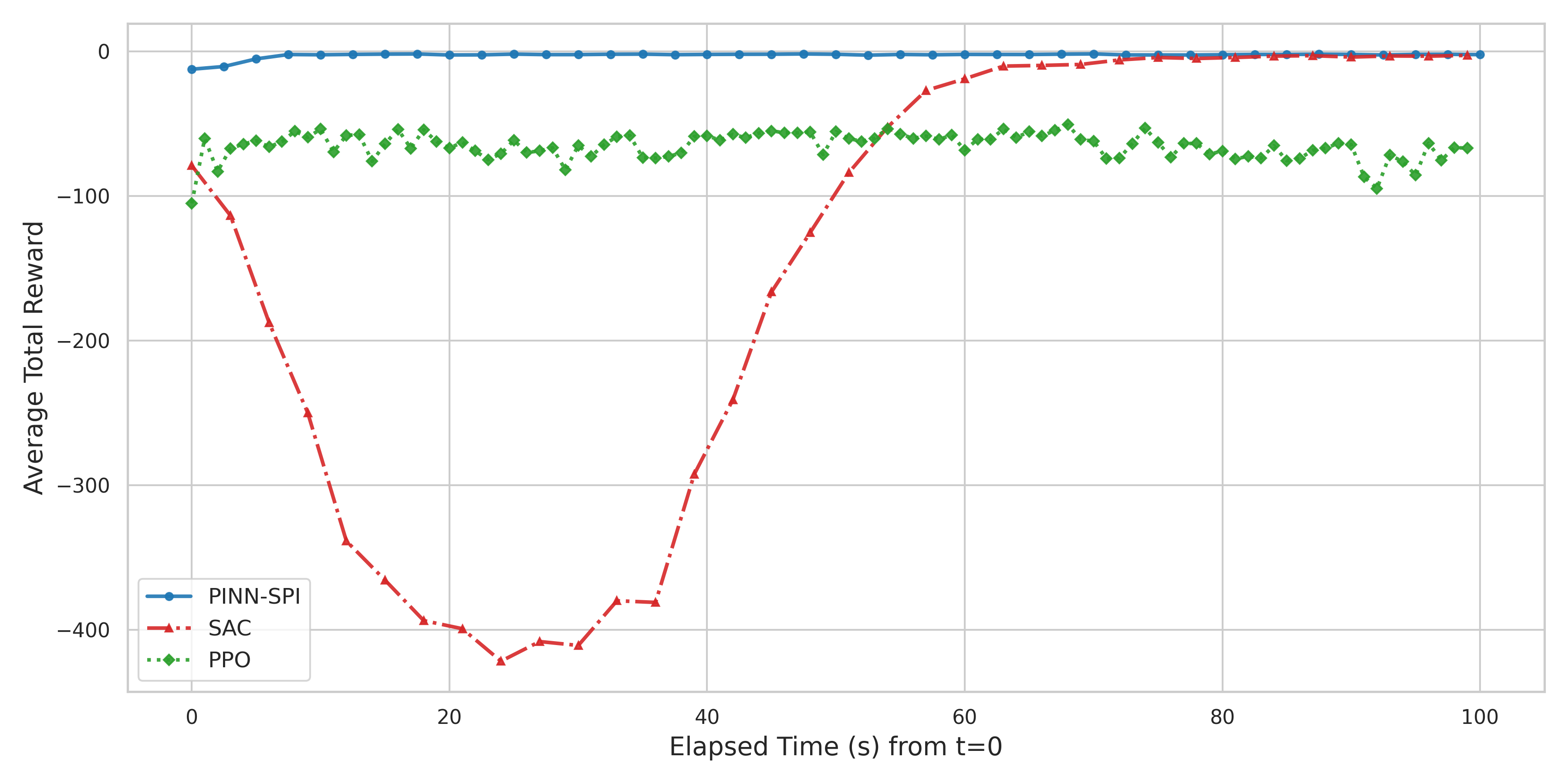}
    \caption{10D stochastic LQR with compact control set and entropy regularization.}
    \label{fig:img1-2}
  \end{subfigure}
  \caption{Comparison of PINN-SPI, SAC, and PPO on 5D and 10D stochastic LQR problems with compact action constraints.}
  \label{fig:lqr:comp}
\end{figure}

\begin{figure}[htbp]
  \centering
  \begin{subfigure}{0.4\textwidth}
    \centering
    \includegraphics[width=\textwidth]{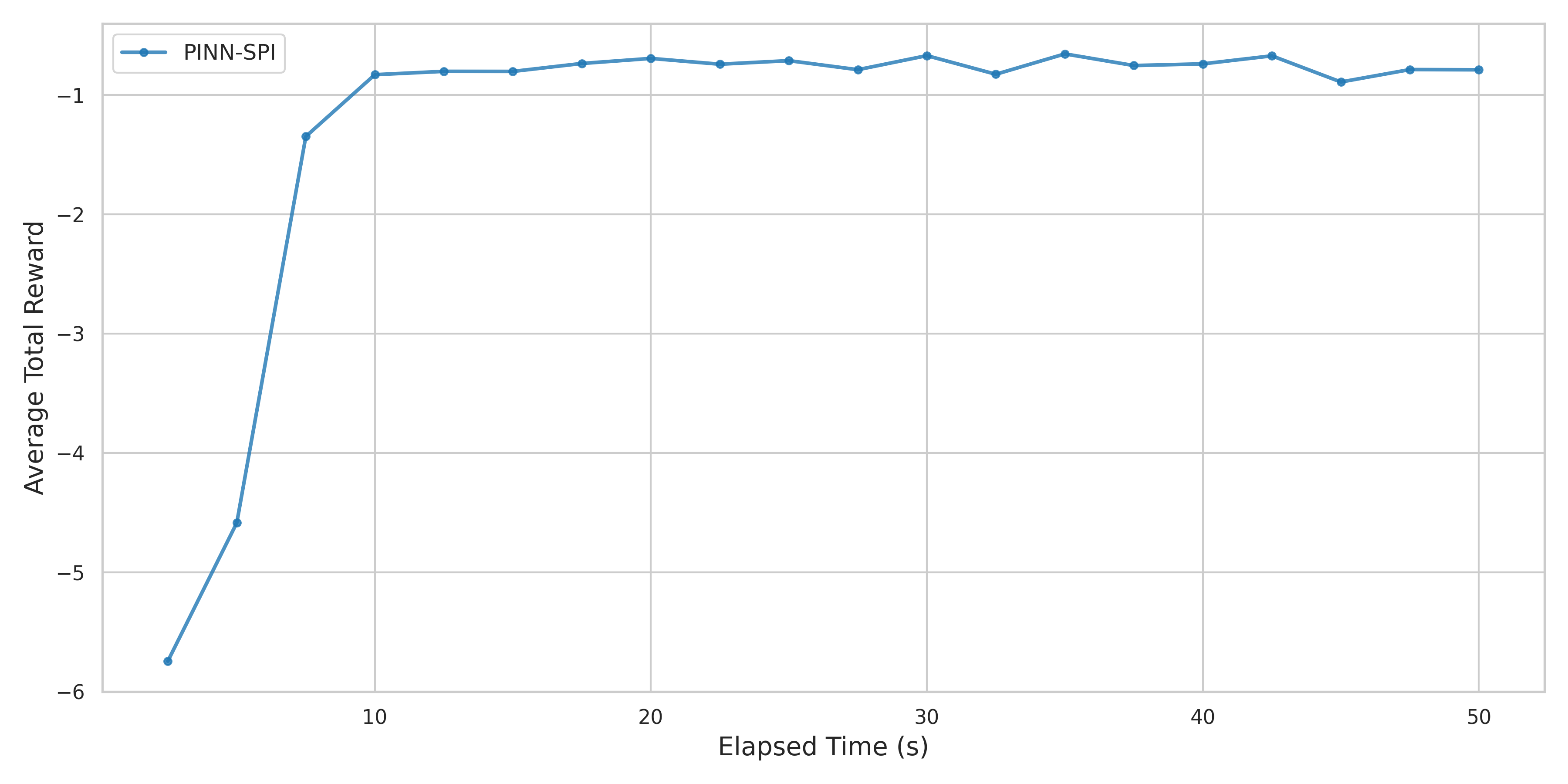}
    \caption{5D stochastic LQR problem.}
    \label{fig:img2-1}
  \end{subfigure}
  \hfill
  \begin{subfigure}{0.4\textwidth}
    \centering
    \includegraphics[width=\textwidth]{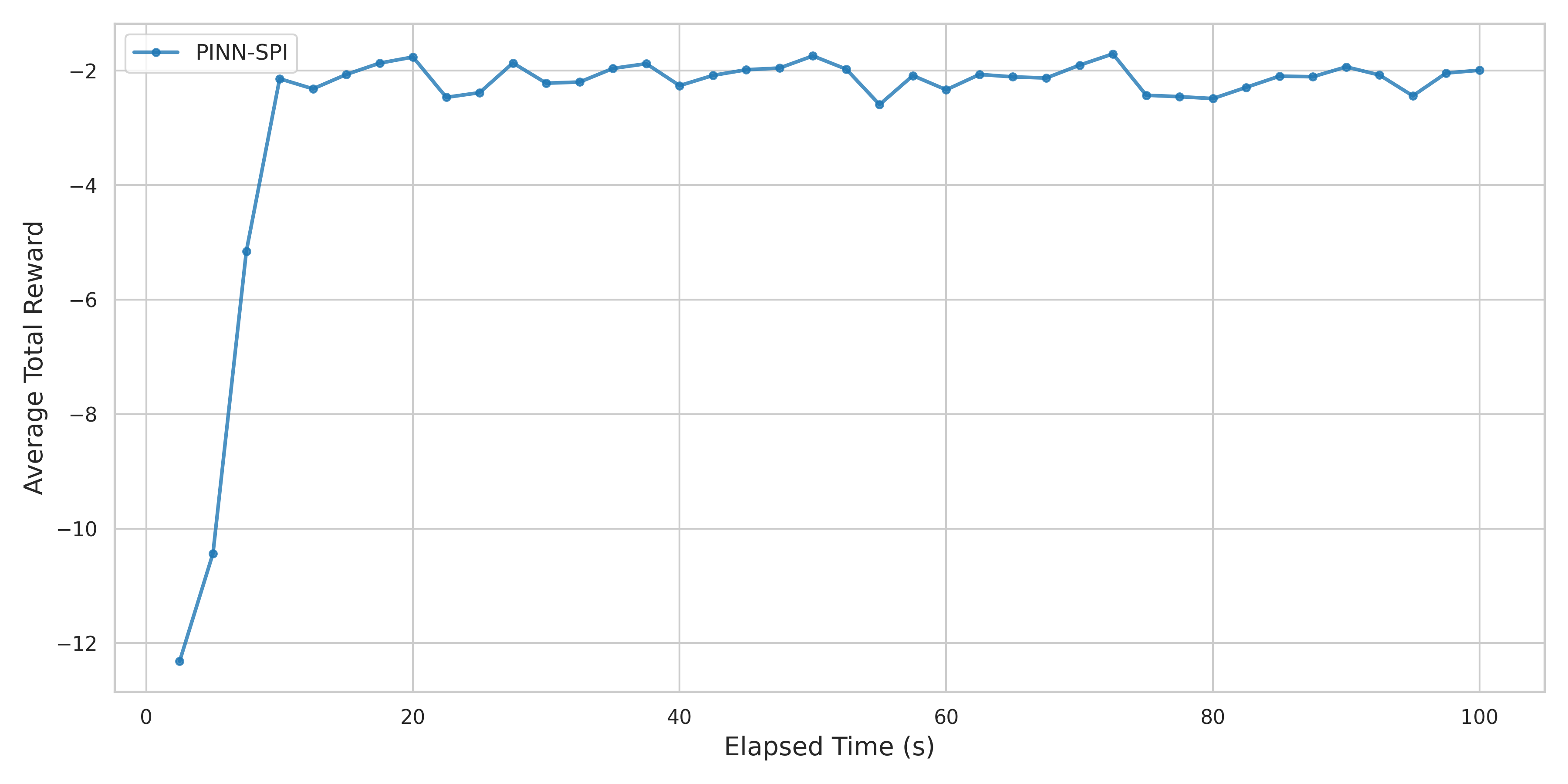}
    \caption{10D stochastic LQR problem.}
    \label{fig:img2-2}
  \end{subfigure}
\caption{Evaluation reward over training time for PINN-SPI on LQR tasks. The average total reward increases monotonically as policy iteration proceeds.}
  \label{fig:lqr:mono}
\end{figure}

\subsection{Nonlinear benchmarks with stochastic dynamics}
To demonstrate performance in more realistic and nonlinear settings, we evaluate our method on two standard control benchmarks: the stochastic inverted pendulum and cartpole. Each system is modeled as a stochastic control-affine system with additive Brownian noise. We adopt a standard stochasticized version of the cartpole and pendulum dynamics supported by OpenAI GYM~\cite{brockman2016openai} with additive noise on state variables.

\paragraph{Setup.} For each environment, we define entropy-regularized cost functionals and apply PINN-SPI with a fixed noise level $\sigma = 0.1  I_d$. We use neural networks for both value approximation and policy extraction, trained via policy iteration with residual loss minimization. 

\paragraph{Result.} Figure~\ref{fig:gym:comp} shows the performance over elapsed time. PINN-SPI achieves faster stabilization and higher reward than SAC and PPO, while maintaining constraint satisfaction. 

\begin{figure}[htbp]
  \centering
  \begin{subfigure}[b]{0.4\textwidth}
    \centering
    \includegraphics[width=\textwidth]{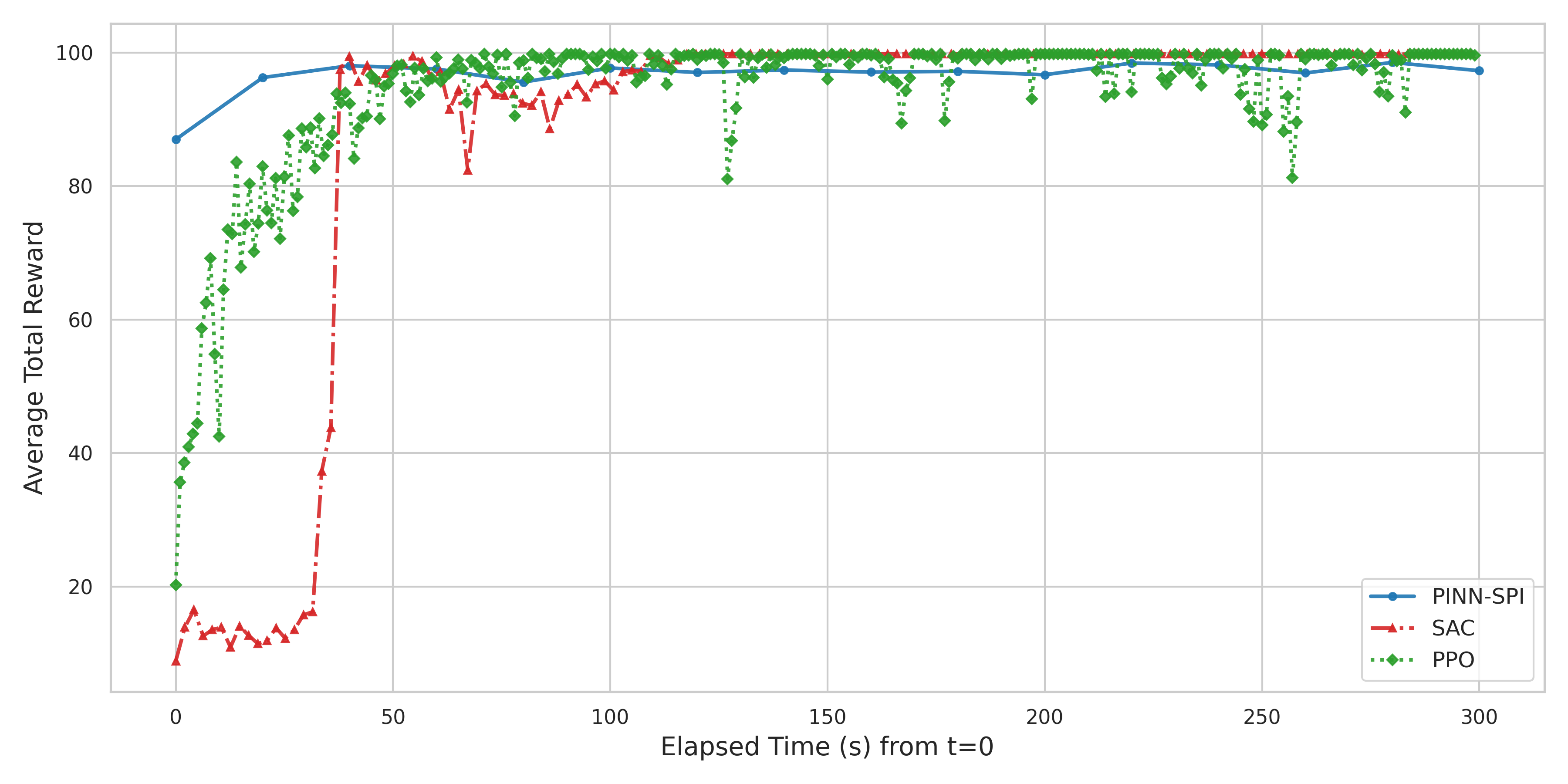}
    \caption{Stochastic cartpole problem.}
    \label{fig:img3-1}
  \end{subfigure}
  \hfill
  \begin{subfigure}[b]{0.4\textwidth}
    \centering
    \includegraphics[width=\textwidth]{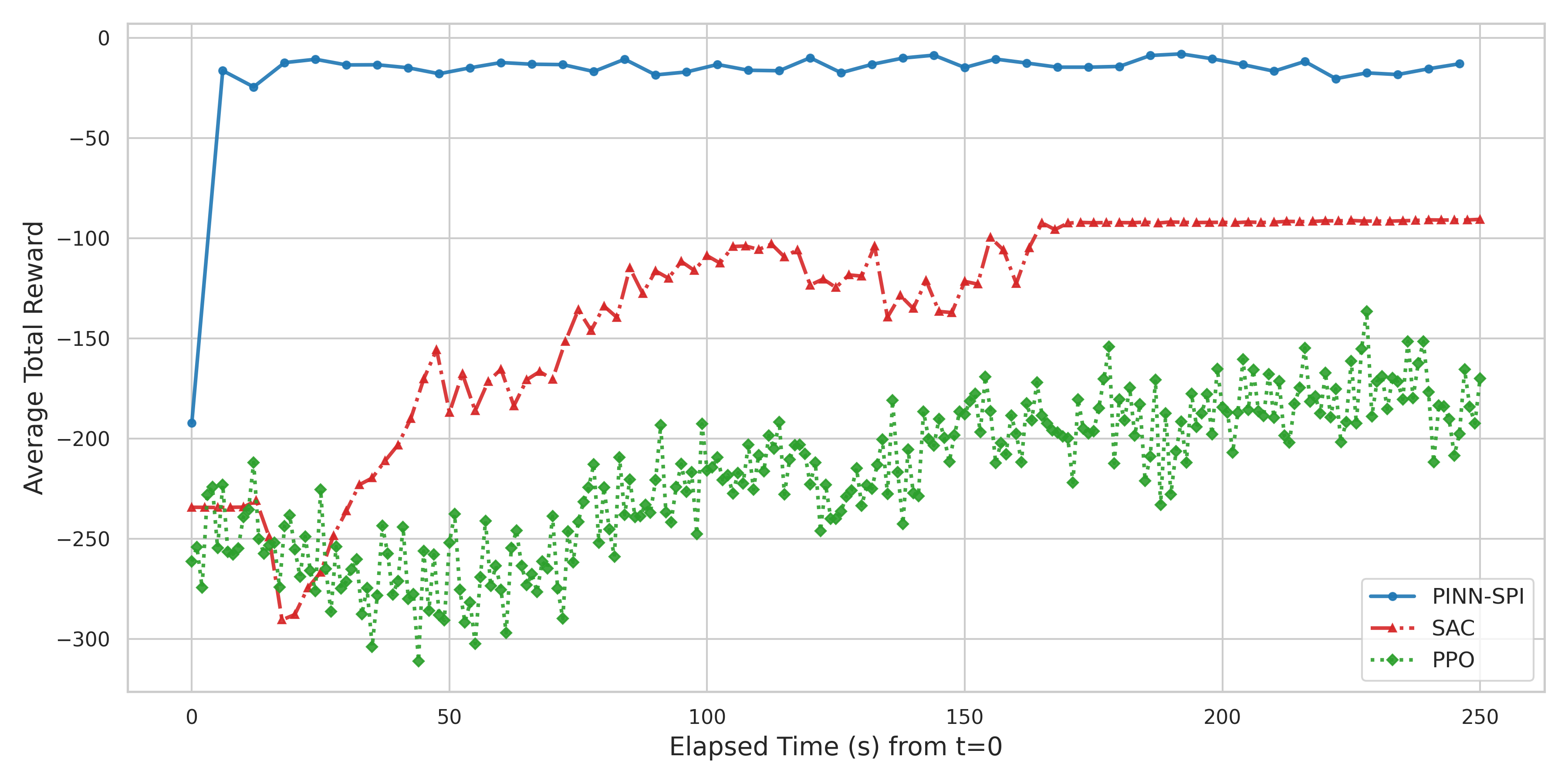}
    \caption{Stochastic pendulum problem.}
    \label{fig:img3-2}
  \end{subfigure}
  \caption{Comparison between PINN-SPI, SAC, and PPO on the cartpole and pendulum problems.}
  \label{fig:gym:comp}
\end{figure}

\section{Discussion}
Our proposed PINN-based soft policy iteration (PINN-SPI) framework provides a structured and scalable method for solving entropy-regularized stochastic control problems. Compared to classical policy iteration, the entropy-regularized formulation enables several theoretical and practical advantages, particularly when combined with mesh-free residual minimization. Our method offers several key advantages. 

\paragraph{Soft updates enable stability.}
In standard policy iteration (e.g., Howard’s method~\cite{howard1960dynamic,KerimkulovSiskaSzpruch2020}), the policy improvement step requires a pointwise maximization, which can result in discontinuous or unstable policies in nonlinear or high-dimensional systems. By contrast, the softmax-based update in exploratory control is differentiable and admits explicit $L^2$-Lipschitz continuity with respect to the value function gradient. This structure enables more stable policy improvement and facilitates gradient-based learning.

\paragraph{Systematic error decomposition.}
The policy evaluation step in our method solves a linear PDE with frozen coefficients. This linearity allows a rigorous $L^2$ error decomposition into three sources: iteration error, residual error, and policy approximation error. Theorem~\ref{thm:expanded} establishes that total approximation error remains uniformly bounded and does not accumulate across iterations, which is a critical stability guarantee for learning-based control methods.

\paragraph{Scalability to nonlinear and high-dimensional settings.}
Our framework applies to general nonlinear stochastic systems with compact control constraints. By using physics-informed neural networks (PINNs) and mesh-free residual minimization, we avoid spatial discretization and enable tractable approximation even in high dimensions. 

We now address some limitations and future directions. 
\paragraph{Policy differentiability assumption.}
Our theoretical results rely on the assumption that policies are differentiable functions of the value gradient (e.g., via softmax). This excludes problems with non-smooth or bang-bang optimal policies, or discrete action spaces, where such structure may not exist.

\paragraph{Model assumptions.}
Our framework assumes full knowledge of system dynamics (drift and diffusion). Extension to model-uncertain settings or learning dynamics jointly with control remains a promising but challenging direction.

\paragraph{Summary.}
The proposed PINN-SPI method bridges classical policy iteration theory and modern neural approximation by leveraging the analytic structure of entropy-regularized control. The result is a mesh-free, provably convergent algorithm that performs well across a broad range of benchmarks. Future work will aim to relax model assumptions, improve policy training stability, and extend to partially observed or data-driven settings.

\section*{Acknowledgements}
This work was supported by Seoul National University of Science and Technology. Yeoneung Kim is supported by the National Research Foundation of Korea (NRF) grant funded by the Korea government (MSIT) (RS-2023-00219980, RS-2023-00211503). Namkyeong Cho was supported by the Gachon University research fund of 2025 (GCU-202502800001). The authors would like to thank Professor Hung Vinh Tran (University of Wisconsin–Madison) for his insightful suggestions and valuable guidance in developing and refining the ideas of this work.

\bibliography{aaai2026}

@article{tran2025policy,
  title={Policy iteration for exploratory Hamilton--Jacobi--Bellman equations},
  author={Tran, Hung Vinh and Wang, Zhenhua and Zhang, Yuming Paul},
  journal={Applied Mathematics \& Optimization},
  volume={91},
  number={2},
  pages={50},
  year={2025},
  publisher={Springer}
}

@article{ma2024convergence,
  title={Convergence analysis for entropy-regularized control problems: A probabilistic approach},
  author={Ma, Jin and Wang, Gaozhan and Zhang, Jianfeng},
  journal={arXiv preprint arXiv:2406.10959},
  year={2024}
}

@book{evans2022partial,
  title={Partial differential equations},
  author={Evans, Lawrence C},
  volume={19},
  year={2022},
  publisher={American mathematical society}
}

@article{wang2020reinforcement,
  title={Reinforcement learning in continuous time and space: A stochastic control approach},
  author={Wang, Haoran and Zariphopoulou, Thaleia and Zhou, Xun Yu},
  journal={Journal of Machine Learning Research},
  volume={21},
  number={198},
  pages={1--34},
  year={2020}
}

@inproceedings{ziebart2008maximum,
  title={Maximum entropy inverse reinforcement learning.},
  author={Ziebart, Brian D and Maas, Andrew L and Bagnell, J Andrew and Dey, Anind K and others},
  booktitle={Aaai},
  volume={8},
  pages={1433--1438},
  year={2008},
  organization={Chicago, IL, USA}
}

@inproceedings{ramesh2023physics,
  title={Physics-informed model-based reinforcement learning},
  author={Ramesh, Adithya and Ravindran, Balaraman},
  booktitle={Learning for Dynamics and Control Conference},
  pages={26--37},
  year={2023},
  organization={PMLR}
}

@article{meng2024physics,
  title={Physics-informed neural network policy iteration: Algorithms, convergence, and verification},
  author={Meng, Yiming and Zhou, Ruikun and Mukherjee, Amartya and Fitzsimmons, Maxwell and Song, Christopher and Liu, Jun},
  journal={arXiv preprint arXiv:2402.10119},
  year={2024}
}

@article{lee2025hamilton,
  title={Hamilton--Jacobi based policy-iteration via deep operator learning},
  author={Lee, Jae Yong and Kim, Yeoneung},
  journal={Neurocomputing},
  pages={130515},
  year={2025},
  publisher={Elsevier}
}

@article{tang2022exploratory,
  title={Exploratory HJB equations and their convergence},
  author={Tang, Wenpin and Zhang, Yuming Paul and Zhou, Xun Yu},
  journal={SIAM Journal on Control and Optimization},
  volume={60},
  number={6},
  pages={3191--3216},
  year={2022},
  publisher={SIAM}
}

@book{csiszar2011information,
  title={Information theory: coding theorems for discrete memoryless systems},
  author={Csisz{\'a}r, Imre and K{\"o}rner, J{\'a}nos},
  year={2011},
  publisher={Cambridge University Press}
}

@article{KerimkulovSiskaSzpruch2020,
    author = {Kerimkulov, B. and {\v{S}}i{\v{s}}ka, D. and Szpruch, {\L}ukasz},
  title = {Exponential Convergence and Stability of Howard's Policy Improvement Algorithm for Controlled Diffusions},
  journal = {SIAM Journal on Control and Optimization},
  year = {2020},
  volume = {58},
  number = {3},
  pages = {1314--1340},
  doi = {10.1137/18M1211650}
}

@article{howard1960dynamic,
  title={Dynamic programming and markov processes.},
  author={Howard, Ronald A},
  year={1960},
  publisher={John Wiley}
}

@article{brockman2016openai,
  title={Open{AI} GYM},
  author={Brockman, Greg and Cheung, Vicki and Pettersson, Ludwig and Schneider, Jonas and Schulman, John and Tang, Jie and Zaremba, Wojciech},
  journal={arXiv preprint arXiv:1606.01540},
  year={2016}
}

@article{schulman2017proximal,
  title={Proximal policy optimization algorithms},
  author={Schulman, John and Wolski, Filip and Dhariwal, Prafulla and Radford, Alec and Klimov, Oleg},
  journal={arXiv preprint arXiv:1707.06347},
  year={2017}
}

@inproceedings{haarnoja2018soft,
  title={Soft actor-critic: Off-policy maximum entropy deep reinforcement learning with a stochastic actor},
  author={Haarnoja, Tuomas and Zhou, Aurick and Abbeel, Pieter and Levine, Sergey},
  booktitle={International conference on machine learning},
  pages={1861--1870},
  year={2018},
  organization={Pmlr}
}

\newpage

\appendix 

\section{Proof of Proposition~\ref{prop:bound}}\label{app:proof:prop}

We next verify that the softmax policy $\pi^n(x,u)$ is uniformly bounded above and below.
By definition,
\[
\pi^n(x,u) = \frac{\exp\left( \frac{1}{\lambda} f(x,u,\nabla_x v^{n-1}(x)) \right)}
                 {\int_U \exp\left( \frac{1}{\lambda} f(x,u',\nabla_x v^{n-1}(x)) \right) \,\mathrm{d}u'},
\]
where $f(x,u,p) := b(x,u)\cdot p + r(x,u).$
Since \( b(x,u), r(x,u) \) are bounded and the value functions $\| v_n \|_{C^2(\R^d)}$ is bounded~\cite{tran2025policy,ma2024convergence}, we obtain a uniform bound on the numerator:
\[
\begin{split}
\exp\left( \tfrac{1}{\lambda} f(x,u,\nabla_x v^{n-1}(x)) \right)\le \exp\left( \tfrac{1}{\lambda} C_{\text{num}} \right).
\end{split}
\]
Meanwhile, the denominator is bounded below by
\[
\begin{split}
&\int_U \exp\left( \tfrac{1}{\lambda} f(x,u',\nabla_x v^{n-1}(x)) \right) ,\ \mathrm{d}u'
\\
&\ge |U| \cdot \exp\left( \tfrac{1}{\lambda} \inf_{u' \in U} f(x,u',\nabla_x v^{n-1}(x)) \right)
\ge C_{\text{den}} > 0.
\end{split}
\]
Hence, the softmax satisfies
\[
0 < \pi^n(x,u) \le \frac{ \exp( \frac{1}{\lambda} C_{\text{num}} ) }{ C_{\text{den}} } =: M < \infty.
\]
By a similar argument, $\pi^n$ is bounded below uniformly. Therefore, for some $m>0$, we have that
\[
\pi^n \in [m, M].
\]

\section{Proof of Lemma~\ref{lem:softmaxLip}}\label{app:lem:softmaxlip}
We first set 
\[
Z(p):=\exp \Bigl[\tfrac1\lambda\bigl(b(x,u) \cdot p+r(x,u)\bigr)\Bigr],
\]
and
\[
D(p):=\int_{U}Z(p,u') \,\mathrm{d}u',
\]
so that $\Phi(p)=Z(p)/D(p)$.

Differentiating $\Phi$ with respect to
$p_{j}$ ($j=1,\dots,d$), we obtain
\[
\partial_{p_{j}}\Phi
=\frac{1}{\lambda D(p)}
       Z(p,u)   \Bigl(b_{j}(x,u)-\bar b_{j}(p)\Bigr),
\]
where
$\bar b_{j}(p):=\frac{\int_{U}b_{j}(x,u')Z(p,u')  \, \mathrm{d}u'}{D(p)}$. Hence, 
\[
\bigl|\partial_{p_{j}}\Phi\bigr|
 =
\Phi \frac{|b_{j}(x,u)-\bar b_{j}(p)|}{\lambda}
\le \frac{2 \sup_{|p|\leq \tilde L} \Phi(p) \|b\|_{L^\infty(\R^d \times U)}}{\lambda}
\]

For arbitrary $p,q\in\R^{d}$ the fundamental theorem of calculus yields
\[
\Phi(p)-\Phi(q)
=\int_{0}^{1}\nabla_{p}\Phi \bigl(q+\theta(p-q)\bigr) (p-q) \,\mathrm{d}\theta .
\]
Taking $L^{2}(U)$norms, and recalling
\[
L_{\phi}:= \frac{2\|b\|_{L^\infty}
  \sup_{|p|\le\tilde L}\Phi(p)
  \sqrt{d}|U|^{1/2}}{\lambda},
\]
we deduce that
\[
\|\Phi(p)-\Phi(q)\|_{L^{2}(U)}
  \le  
\int_{0}^{1}L_{\Phi}   |p-q|   \,\mathrm{d}\theta
=L_{\Phi}   |p-q|,
\]
which is exactly \eqref{eq:Lip-est}.

\section{Proof of Lemma~\ref{lem:local-stability}}\label{lem:lip:proof}
With $\|f\|_{H^1(\mathcal{X})}:=\|f\|_{L^2(\mathcal{X})}+\|\nabla_x f\|_{L^2(\mathcal{X})}$, we note first that $\|v^\pi\|_{H^1(\mathcal{X}))}+\|\tilde v^{\tilde \pi}\|_{H^1(\mathcal{X})}<\infty$ by the elliptic regularity theory~\cite{evans2022partial}. Setting $e:=v^{\tilde\pi}-v^{\pi}$, we have that
\[
\begin{split}
\rho e
 &=b^{\tilde\pi} \cdot\nabla_x e
  +\tfrac12\operatorname{tr}(\Sigma D_{xx}^{2}e)
  +(b^{\tilde\pi}-b^{\pi}) \cdot\nabla_x v^{\pi}
  \\
  &+(f^{\tilde\pi}-f^{\pi})
  -\lambda(\mathcal H^{\tilde\pi}-\mathcal H^{\pi}).
\end{split}
\]
Multiply by $e$ and integrate over $\mathcal{X}$. By Assumption~\ref{ass:decay}.
\[
\begin{split}
 &(\rho-\tfrac12B)\|e\|_{2}^{2}
 +\frac{1}{C_0}\|\nabla_x e\|_{2}^{2}
 \\
 &\le
 (\|b^{\tilde\pi}-b^{\pi}\|_{2}
        \|\nabla_x v^{\pi}\|_{2}
       +\|f^{\tilde\pi}-f^{\pi}\|_{2}
       +\lambda\|\mathcal H^{\tilde\pi}-\mathcal H^{\pi}\|_{2})
        \\
        &\qquad \qquad \qquad \qquad \times \|e\|_{2},
\end{split}
\]
where all norms are over $\mathcal{X}$.  
Estimate the three coefficient differences:
\[
\|b^{\tilde\pi}-b^{\pi}\|_{2}\le C_1
\|\tilde\pi-\pi\|_{L^{2}(\mathcal{X}\times U)},
\]
and
\[
\|f^{\tilde\pi}-f^{\pi}\|_{2}\le C_1
                                \|\tilde\pi-\pi\|_{L^{2}(\mathcal{X}\times U)},
\]
where $C_1$ is from Assumption~\ref{ass:main}.
For the entropy term, we recall Proposition~\ref{prop:bound}, and set $L_H:= \sup_{s\in [m,M]} |1+\log s|$ where $m,M$ are from the proposition to obtain
\[
\|\mathcal{H}^{\pi} - \mathcal{H}^{\tilde \pi}\|_{2}
\le L_H\|\tilde \pi-\pi\|_{L^2 (\mathcal{X} \times U)}.
\]
Combining together, we have
\[
\begin{split}
&(\rho-\tfrac12B)\|e\|_{2}^{2}
 +\frac{1}{2C_0}\|\nabla_x e\|_{2}^{2}\\
 &\qquad \le \underbrace{(C_1C_2+C_1+\lambda C_1)L_H}_{=:C}\|\tilde \pi-\pi\|_{L^2(\mathcal{X} \times U)}\|e\|_2.
\end{split}
\]
Invoking the Young's inequality $ab \leq  \epsilon a^2 +\frac{1}{4\epsilon}b^2$ with $a=\|e\|_2^2$, $b=\|\tilde\pi-\pi\|_{L^(\mathcal X \times U)}$ and $\epsilon=\frac{\rho-\frac{1}{2}B}{2C}$, we deduce that
\[
\begin{split}
&\frac{\rho-\frac{1}{2}B}{2}\|e\|_2^2 +\frac{1}{2C_0}\|\nabla_x e\|_2^2 \\
&\qquad\leq {\frac{C}{2(\rho-\frac{1}{2}B)}} \|\tilde\pi-\pi\|^2_{L^2(\mathcal{X} \times U)}.
\end{split}
\]

Therefore,
\[
\|e\|_2  \leq \frac{\sqrt{C}}{\rho-\frac{1}{2}B} \leq \|\tilde\pi-\pi\|_{L^2(\mathcal{X} \times U)}
\]
and
\[
\|\nabla_x e\|_2 \leq C\sqrt{\frac{ C_0}{\rho-\frac{1}{2}B}} \|\tilde\pi-\pi\|_{L^2(\mathcal{X}\times U)}.
\]
To get the desired result, we choose
\[
\tilde C_\rho:=\max\{\frac{\sqrt{C}}{\rho-\frac{1}{2}B},C\sqrt{\frac{ C_0}{\rho-\frac{1}{2}B}} \}.
\]

\section{Experimental details}

\subsection{Stochastic linear-quadratic regulator.}

For the stochastic LQR problem where the dynamics is given by
\[
\mathrm{d}X_t=(AX_t+Bu_t)\mathrm{d} t + \sigma\mathrm{d}W_t,
\]
we use
\[
A = \begin{bmatrix}
0.0446 & 0.0300 & 0.4497 & 0.2813 & 0.0311 \\
0.0432 & 0.2885 & 0.0321 & 0.0252 & 0.0419 \\
0.1975 & 0.0430 & 0.0301 & 0.0422 & 0.2282 \\
0.0436 & 0.4716 & 0.3812 & 0.3010 & 0.0372 \\
0.0225 & 0.0629 & 0.0335 & 0.0424 & 0.1836 \\
\end{bmatrix},\]

\[
B = \begin{bmatrix}
0.0681 & 0.0544 & 0.0467 & 0.0152 & 0.0787 \\
0.0970 & 0.0081 & 0.0145 & 0.0034 & 0.0984 \\
0.0358 & 0.0833 & 0.0324 & 0.0839 & 0.0012 \\
0.0116 & 0.0280 & 0.0056 & 0.0092 & 0.0432 \\
0.0047 & 0.0848 & 0.0718 & 0.0977 & 0.0556 \\
\end{bmatrix}
\]
for $d=m=5$, and
\[
\resizebox{\columnwidth}{!}{$
A = \begin{bmatrix}
0.0866 & 0.0191 & 0.1394 & 0.1392 & 0.1480 & 0.0241 & 0.0988 & 0.1435 & 0.1883 & 0.1923 \\
0.1271 & 0.0900 & 0.0924 & 0.0766 & 0.0696 & 0.0466 & 0.1716 & 0.0371 & 0.1097 & 0.0946 \\
0.0595 & 0.1656 & 0.1953 & 0.1353 & 0.1872 & 0.0587 & 0.0830 & 0.0035 & 0.0215 & 0.0740 \\
0.1971 & 0.0808 & 0.1301 & 0.0157 & 0.1908 & 0.1505 & 0.0662 & 0.1334 & 0.1394 & 0.1951 \\
0.0570 & 0.0419 & 0.0470 & 0.0916 & 0.1094 & 0.0640 & 0.0159 & 0.1687 & 0.1224 & 0.0294 \\
0.1137 & 0.1033 & 0.0379 & 0.0881 & 0.1224 & 0.0139 & 0.0060 & 0.1857 & 0.0732 & 0.0989 \\
0.1271 & 0.0414 & 0.1232 & 0.1896 & 0.1457 & 0.0997 & 0.1830 & 0.1309 & 0.0673 & 0.0855 \\
0.0827 & 0.1076 & 0.1498 & 0.1164 & 0.0192 & 0.1888 & 0.1357 & 0.1352 & 0.1086 & 0.1959 \\
0.1489 & 0.0223 & 0.0018 & 0.0002 & 0.1631 & 0.1272 & 0.0282 & 0.0075 & 0.0351 & 0.0478 \\
0.0791 & 0.1215 & 0.0219 & 0.1653 & 0.0635 & 0.0230 & 0.1943 & 0.0373 & 0.0253 & 0.1129
\end{bmatrix},$}\]
\[
\resizebox{\columnwidth}{!}{$
B = \begin{bmatrix}
0.0728 & 0.0209 & 0.0336 & 0.0269 & 0.0197 & 0.0844 & 0.0358 & 0.0483 & 0.0033 & 0.0559 \\
0.0459 & 0.0812 & 0.0411 & 0.0338 & 0.0581 & 0.0689 & 0.0084 & 0.0015 & 0.0198 & 0.0053 \\
0.0136 & 0.0803 & 0.0666 & 0.0570 & 0.0277 & 0.0721 & 0.0613 & 0.0106 & 0.0115 & 0.0699 \\
0.0848 & 0.0815 & 0.0915 & 0.0766 & 0.0683 & 0.0997 & 0.0389 & 0.0485 & 0.0630 & 0.0102 \\
0.0858 & 0.0205 & 0.0769 & 0.0968 & 0.0722 & 0.0004 & 0.0201 & 0.0990 & 0.0836 & 0.0750 \\
0.0168 & 0.0033 & 0.0286 & 0.0740 & 0.0314 & 0.0238 & 0.0183 & 0.0277 & 0.0889 & 0.0123 \\
0.0795 & 0.0445 & 0.0037 & 0.0776 & 0.0038 & 0.0103 & 0.0183 & 0.0542 & 0.0722 & 0.0544 \\
0.0399 & 0.0139 & 0.0071 & 0.0227 & 0.0556 & 0.0885 & 0.0062 & 0.0271 & 0.0382 & 0.0641 \\
0.0142 & 0.0063 & 0.0456 & 0.0536 & 0.0993 & 0.0206 & 0.0264 & 0.0463 & 0.0695 & 0.0907 \\
0.0301 & 0.0514 & 0.0583 & 0.0007 & 0.0900 & 0.0426 & 0.0385 & 0.0077 & 0.0110 & 0.0930
\end{bmatrix}
$}
\]
for $d=m=10$. The reward function was set to
\[
L(x, u) = -x^\top Q x - u^\top R u,
\]
with compact action set $U = \{u \in \R^m : \|u\|_\infty \leq \bar u\}$.
We set $Q = 5 I_d$, $R = I_m$, $\bar u = 10$, and use isotropic noise
$\sigma = 0.1 I_d$. The entropy weight is fixed to $\lambda = 0.1$.
The total return is computed as the mean over 30 trajectories with
time step $\Delta t = 0.01$, horizon $T = 2$, and discount rate $\rho = 1$.


\paragraph{PINN–SPI.}
For all dimensions $d \in \{5,10,20\}$, we use the same network
architecture and hyperparameters.
The \emph{value network} $v(x;\theta)$ is an MLP with
hidden width $100$ and depth $2$, SiLU activations, and residual
connections between hidden layers. It is optimized by Adam with
learning rate $3\times 10^{-4}$.
The \emph{policy network} $\pi(x,u;\omega)$ is an MLP with hidden
width $100$ and depth $4$, also using SiLU and residual connections,
trained by Adam with learning rate $3\times 10^{-5}$.
At each policy-iteration step we draw $100$ state collocation points
and, for every state, approximate the entropy terms using
$100$ action samples on the compact control set.
We perform $1000$ outer policy-iteration steps, and in each step
run $10$ epochs of value-network training and $10$ epochs of
policy-network training.


\paragraph{SAC.}
For SAC, both the actor and critic are MLPs with hidden width
$256$, depth $3$, ReLU activations, and output layers adapted to
the action space. All networks are trained with Adam with learning
rate $1\times 10^{-5}$. We use a replay buffer with mini-batch size
$256$, Polyak averaging coefficient $\tau = 0.005$, and entropy
regularization coefficient $0.2$. Training is performed over
$2000$ gradient-update iterations, with $500$ warm-up steps before
starting policy updates.

\paragraph{PPO.}
For PPO, the actor and critic share the same MLP architecture:
hidden width $256$, depth $3$, ReLU activations, trained with Adam.
We use clipping parameter $0.2$, generalized advantage estimation
$\lambda_{\mathrm{GAE}} = 0.95$, value loss coefficient $0.5$,
entropy coefficient $0.01$, and gradient clipping with
$\texttt{max\_grad\_norm} = 1.0$.
For the $5$D LQR problem, we set the actor and critic learning
rates to $3\times 10^{-5}$, use $4096$ rollout steps per update,
$10$ PPO epochs, and mini-batch size $256$.
For the $10$D problem, the actor and critic learning rates are
$2\times 10^{-5}$ and $6\times 10^{-5}$, respectively, with
$1024$ rollout steps, $6$ epochs, and mini-batch size $512$.

\subsection{Stochastic Cartpole}

We consider the stochastic cartpole environment with discount
factor $\rho = 0.5$ and entropy temperature $\lambda = 0.1$.
The total reward is computed as the average over $10$ trajectories.

\paragraph{PINN–SPI.}
Since the action space is discrete, we do not parameterize a
separate policy network. Instead, we apply the softmax update
in~(4) and act greedily by selecting the action with the highest
reward. The value function $v(x;\theta)$ is represented
by an MLP with hidden width $100$, depth $2$, SiLU activations,
and residual connections, trained by Adam with learning rate
$3\times 10^{-4}$.
PINN–SPI is run for $30$ outer iterations; in each iteration
we perform $3000$ gradient steps using $10000$ state collocation
points. For each state we evaluate all available actions, which is
equivalent to using one action sample per state.

\paragraph{SAC.}
The SAC actor and critic are MLPs with hidden width $256$,
depth $3$, ReLU activations, trained with Adam with learning rate
$3\times 10^{-4}$. We use mini-batch size $256$, Polyak update
coefficient $\tau = 0.005$, entropy coefficient $0.2$, and
$300$ warm-up steps. Training is run for $200$ iterations.

\paragraph{PPO.}
The PPO actor and critic are again $256$-wide, $3$-layer MLPs with
ReLU activations and trained with Adam optimizer with learning rate $10^{-4}$.
We use $1024$ rollout steps per update, mini-batch size $256$,
$10$ PPO epochs, clipping parameter $0.2$, GAE parameter
$\lambda_{\mathrm{GAE}} = 0.95$, value loss coefficient $0.5$,
entropy coefficient $0.01$, and maximum gradient norm $1.0$.

\subsection*{Stochastic Pendulum}

For the stochastic pendulum, we use discount factor $\rho = 0.5$
and report the average total reward over $10$ trajectories.

\paragraph{PINN–SPI.}
The value network $v(x;\theta)$ is an MLP with hidden width $100$
and depth $2$, using SiLU activations and residual connections,
trained by Adam with learning rate $3\times 10^{-4}$.
The policy network $\pi(x,u;\omega)$ has hidden width $100$ and
depth $4$, also with SiLU and residual connections, trained with
Adam with learning rate $3\times 10^{-5}$.
PINN–SPI is executed for $12$ outer iterations; in each iteration
we perform $1000$ gradient steps for both value and policy
networks. We use $100$ state collocation points for the value
network and $1000$ state--action collocation points for the policy
network, corresponding to $1000$ action samples per state on the
continuous action domain. The entropy weight is fixed to
$\lambda = 0.1$.

\paragraph{SAC.}
The SAC actor and critic are $256$-wide, 3-layer MLPs with ReLU
activations, trained with Adam with learning rate $3\times 10^{-4}$.
We use a mini-batch size of $256$, entropy coefficient $0.2$,
Polyak coefficient $\tau = 0.005$, and run $20\,000$ gradient
iterations.

\paragraph{PPO.}
For PPO, the actor and critic are again $256$-wide, 3-layer MLPs
with ReLU activations, trained with Adam with learning rate
$10^{-4}$ for both networks. We use $1024$ rollout steps, mini-batch
size $256$, $10$ PPO epochs per update, clipping parameter $0.2$,
GAE parameter $\lambda_{\mathrm{GAE}} = 0.95$, value loss coefficient
$0.5$, entropy coefficient $0.01$, and maximum gradient norm $1.0$.


\section{Additional experimental result}\label{app:exp}

\subsection{20D stochastic LQR}
Consistent with the 5D and 10D LQR settings, the 20D stochastic LQR experiment also exhibits a clear upward trend in the cumulative reward as training progresses, in agreement with our theoretical convergence results. This demonstrates the scalability of our proposed approach to higher-dimensional stochastic control problems.

For the 20D system, the matrices $A$ and $B$ are generated by sampling each entry independently from the uniform distribution on $[0, 0.5]$, following the same protocol used in the 5D and 10D experiments. We employ identical experimental configurations for PINN--SPI and SAC as in the lower-dimensional cases. For PPO, we adopt a slightly modified setup: both the actor and critic learning rates are set to $3\times 10^{-5}$, and we use $512$ rollout steps per update, $4$ training epochs, and a mini-batch size of $256$.

Figure~\ref{fig:20d_lqr} illustrates that our method yields competitive performance and consistently attains higher cumulative rewards throughout training.

\begin{figure}[ht]   
    \includegraphics[width=.4\textwidth]{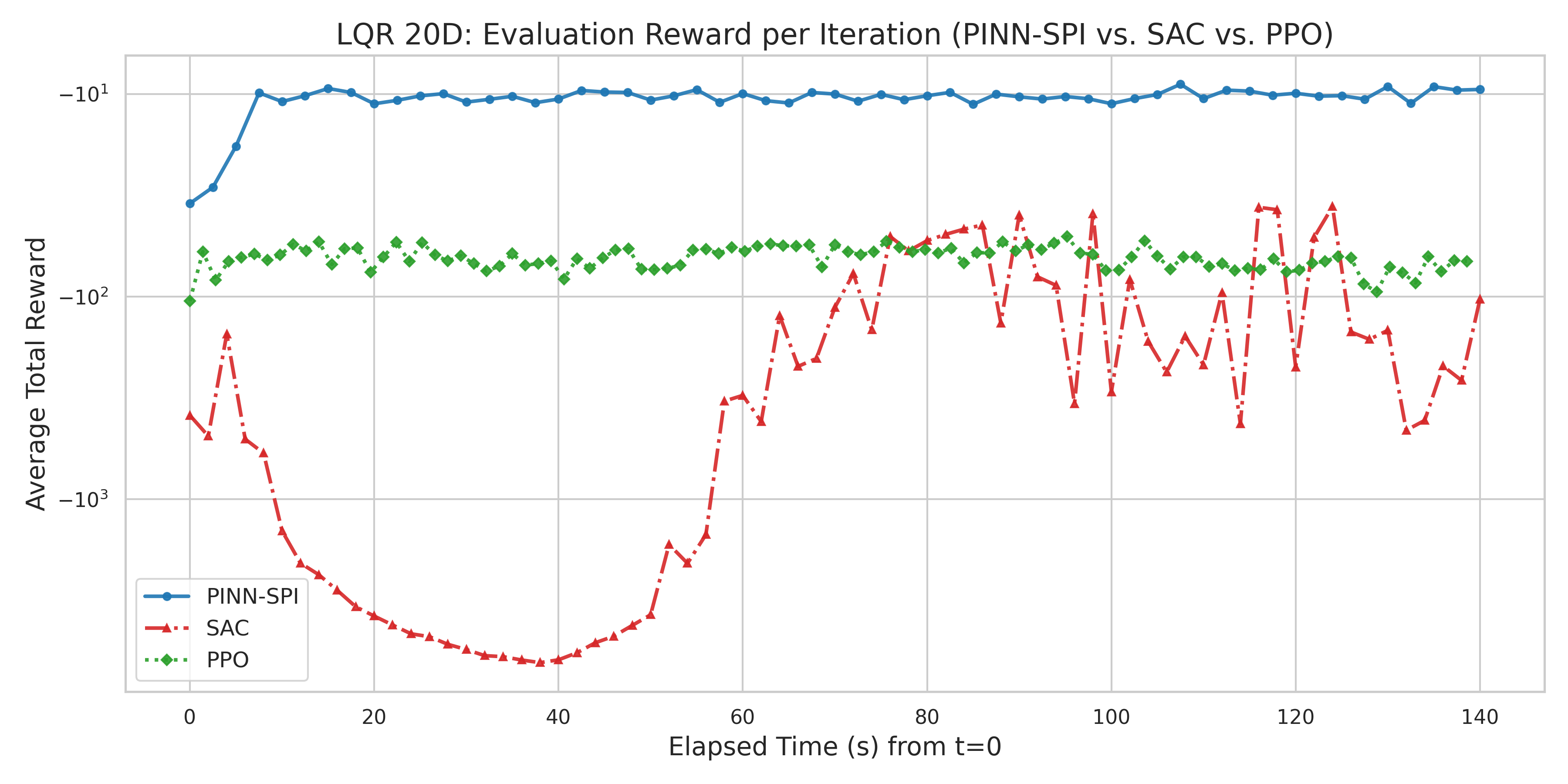}
    \caption{Comparison between our method and others, SAC and PPO.}
    \label{fig:20d_lqr}
\end{figure}

\subsection{Verification of monotonicity}
A key property of soft policy iteration~\cite{tran2025policy} is the monotonic improvement of the value function under exact updates. Our PINN--SPI method numerically preserves this monotonicity despite the presence of approximation errors from neural representation and finite collocation as shown in Figure~\ref{fig:cartpole:mono}

\begin{figure}[ht]
  \centering
  \begin{subfigure}{0.4\textwidth}
    \centering
    \includegraphics[width=\textwidth]{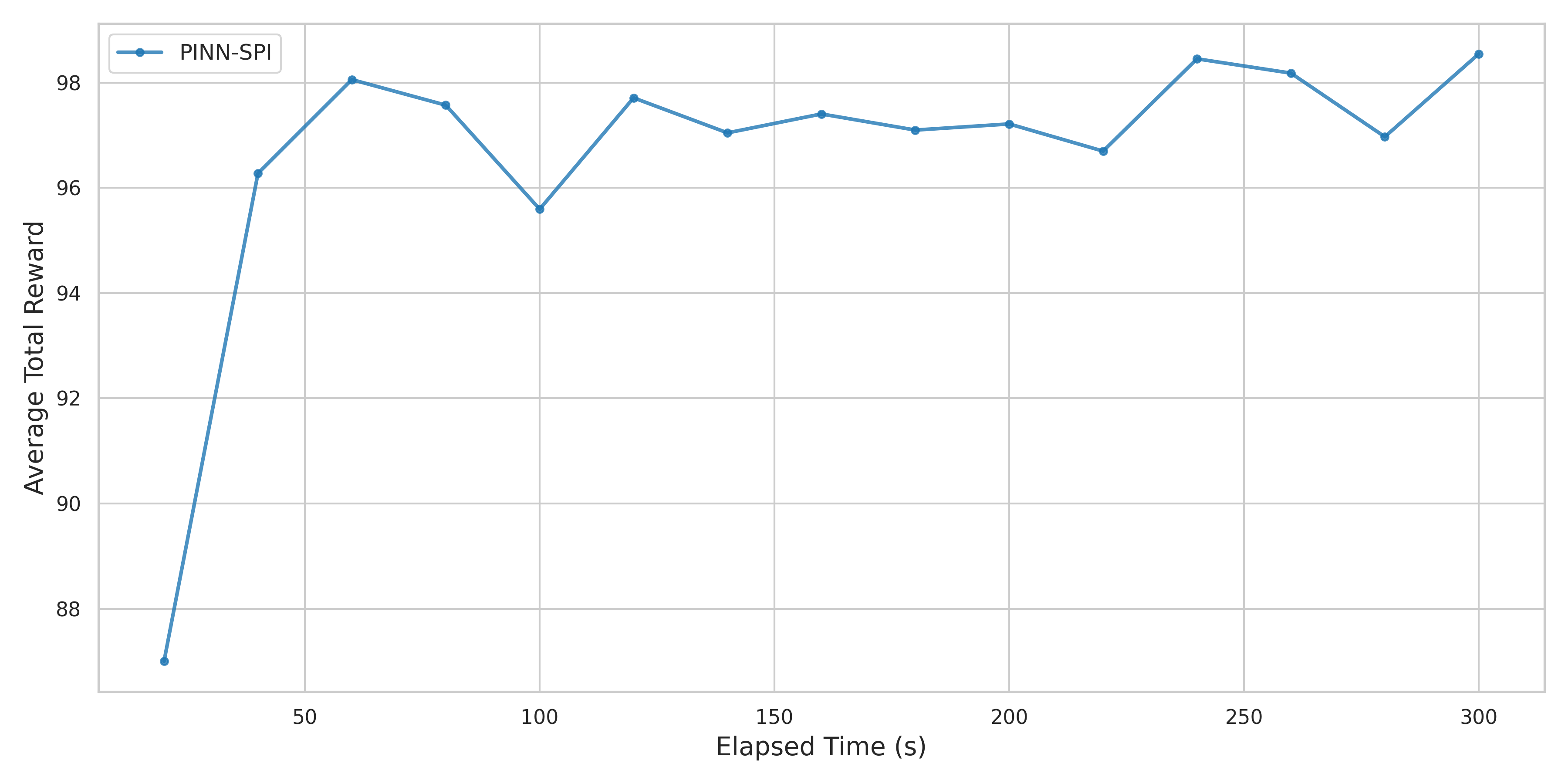}
    \caption{Cartpole: Stable and monotonic convergence.}
    \label{fig:img2-1}
  \end{subfigure}
  \hfill
  \begin{subfigure}{0.4\textwidth}
    \centering
    \includegraphics[width=\textwidth]{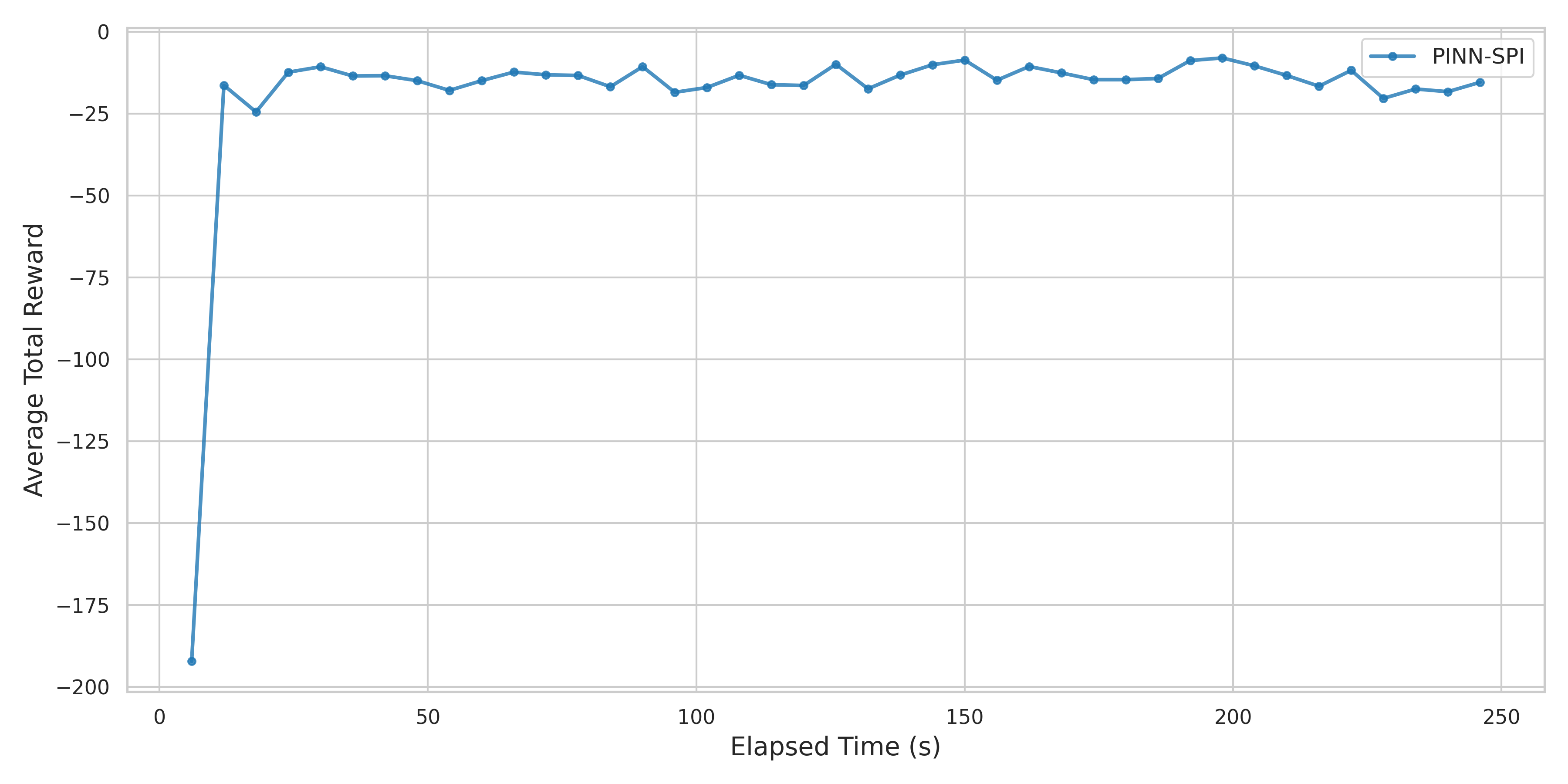}
    \caption{Pendulum: PINN-SPI reaches high total reward.}
    \label{fig:img2-2}
  \end{subfigure}
\caption{Evaluation reward over training time for PINN-SPI on cartpole and cendulum tasks. The average total reward tends to increase as policy iteration proceeds.}
  \label{fig:cartpole:mono}
\end{figure}

\end{document}